\documentclass[10pt,a4paper]{article}
\usepackage[utf8]{inputenc}
\usepackage[T1]{fontenc}
\usepackage{lmodern}
\usepackage{amsmath, amsfonts, amssymb, amsthm}
\usepackage{geometry}
\usepackage{graphicx}
\usepackage{subcaption}
\usepackage{xcolor}
\usepackage{hyperref}
\usepackage{bbm}
\usepackage{caption}
\usepackage{lineno}
\usepackage{eurosym}
\usepackage{comment}
\usepackage{algorithm}
\usepackage{algpseudocode}
\usepackage{float}
\usepackage{amsthm}
\usepackage{booktabs}
\usepackage{amsthm}

\newtheorem{theorem}{Theorem}[section]
\newtheorem{lemma}[theorem]{Lemma}
\newtheorem{proposition}[theorem]{Proposition}
\newtheorem{definition}[theorem]{Definition}
\newtheorem{remark}[theorem]{Remark}

\newtheorem{problem}[theorem]{Problem}

\begin{document}
	
	\title{The Inverse Problem of Velocity Reconstruction from Flow-Induced
		Magnetic Fields: Stability and Uniqueness}
	\author{Yacine Mokhtari}
	\date{}
	\maketitle
	
	\begin{abstract}
		The reconstruction of velocity fields from flow-induced magnetic
		measurements is a fundamental inverse problem in magnetohydrodynamics, with
		broad applications in geophysics, astrophysics, and metallurgy. To address
		the nonlinearity of the governing equations, the magnetic induction equation
		is linearized around a known, space-dependent background magnetic field.
		This reveals a coupled parabolic--hyperbolic system in which magnetic
		diffusion interacts with velocity advection along the background field
		lines. By decoupling these mechanisms into a two-stage framework, sufficient
		conditions are established for both the uniqueness and stability of the
		velocity recovery. The main result proves that the non-uniqueness arising
		from unobservable velocity perturbations parallel to the background field is
		eliminated when the velocity trace on the inflow boundary is prescribed
		alongside magnetic observations. The analysis yields global uniqueness and a
		quantitative stability estimate of hybrid type: a Lipschitz contribution
		from the inflow trace combined with a weaker modulus from the parabolic
		observation, reflecting the mixed parabolic--hyperbolic structure of the
		problem. These theoretical findings are validated through three-dimensional
		numerical simulations that demonstrate the robustness of the decoupled
		reconstruction approach under noisy conditions.
	\end{abstract}
	
	
	
	\section{Introduction}
	
	The motion of a conducting fluid through a magnetic field induces electric
	currents, which in turn generate secondary magnetic fields. This
	magneto-induction mechanism is described by the magnetic induction equation
	and underlies a wide range of phenomena in geophysics, astrophysics, space
	physics, and industrial applications; see, for example, \cite%
	{davidson2001introduction,goedbloed2019magnetohydrodynamics,gurnett2017introduction,gerbeau2006mathematical,zhdanov2002geophysical}%
	. In many of these settings, direct measurements of the fluid velocity are
	difficult or impractical because of the inaccessibility of the flow, the
	opacity of the medium, or the extreme physical conditions. Magnetic
	measurements therefore provide a potentially valuable indirect means of
	recovering the underlying fluid motion. This motivates the inverse problem
	of reconstructing a velocity field from magnetic field observations, which
	has received considerable attention in the literature; see, e.g., \cite%
	{stefani1999velocity,stefani2000uniqueness,lehtikangas2016reconstruction,mokhtari2026velocity}%
	.
	
	The evolution of electromagnetic fields in a conducting medium is governed
	by Maxwell's equations. Under the magnetohydrodynamic approximation,
	neglecting displacement currents and free-charge accumulation, and using
	Ohm's law 
	\begin{equation*}
		\boldsymbol{J} =\sigma\bigl(\boldsymbol{E}+\boldsymbol{v}\times\boldsymbol{B}%
		\bigr),
	\end{equation*}
	one obtains the magnetic induction equation 
	\begin{equation}
		\partial_t\boldsymbol{B} -\eta\Delta\boldsymbol{B} =\nabla\times(\boldsymbol{%
			v}\times\boldsymbol{B}), \qquad \nabla\cdot\boldsymbol{B}=0, \qquad \eta=%
		\frac{1}{\mu_0\sigma},  \label{eq:induction}
	\end{equation}
	where $\boldsymbol{v}$ denotes the fluid velocity, $\boldsymbol{B}$ the
	magnetic induction, $\sigma>0$ the electrical conductivity, $\mu_0>0$ the
	magnetic permeability, and $\eta>0$ the magnetic diffusivity; see \cite%
	{stefani1999velocity,stefani2000uniqueness}. Throughout this work, we set $%
	\eta=1$ by nondimensionalization.
	
	We consider a connected bounded domain $\Omega\subset\mathbb{R}^3$ with sufficiently
	smooth boundary $\partial\Omega$, and a fixed time $T>0$. The magnetic field
	is subject to the boundary conditions 
	\begin{equation}
		\boldsymbol{B}\cdot\boldsymbol{n}=g_1, \qquad \boldsymbol{n}%
		\times(\nabla\times\boldsymbol{B}) =\boldsymbol{g}_2 \quad\text{on }%
		(0,T)\times\partial\Omega,  \label{eq:BC-magnetic}
	\end{equation}
	where $\boldsymbol{n}$ denotes the outward unit normal, $g_1$ is a
	scalar-valued boundary datum, and $\boldsymbol{g}_2$ is a tangential
	vector-valued datum. The first condition prescribes the normal component of
	the magnetic field and hence the magnetic flux through the boundary. The
	second prescribes the tangential component of its curl. Since 
	\begin{equation*}
		\nabla\times\boldsymbol{B}=\mu_0\boldsymbol{J},
	\end{equation*}
	the latter condition is equivalent to prescribing the tangential component
	of the current density, up to the factor $\mu_0$. These boundary conditions
	are natural in the analysis of magnetic diffusion operators and their
	associated boundary-value problems; see, for example, \cite%
	{schwarz1995hodge,arnold2006finiteelement}. In the Hodge-theoretic
	formulation, the homogeneous conditions 
	\begin{equation*}
		\boldsymbol{B}\cdot\boldsymbol{n}=0, \qquad \boldsymbol{n}\times(\nabla\times%
		\boldsymbol{B})=\boldsymbol{0} \quad\text{on }\partial\Omega
	\end{equation*}
	are the magnetic (absolute) boundary conditions for the Hodge Laplacian on
	divergence-free vector fields. Their physical interpretation depends on the
	electromagnetic model and the assumptions imposed on the wall; they should
	therefore not be identified indiscriminately with every formulation of a
	perfectly conducting boundary. For the latter distinction and the associated
	MHD boundary-value formulations, see \cite%
	{gerbeau2006mathematical,freidberg1987ideal}. The boundary data must also
	satisfy the compatibility conditions imposed by the divergence constraint
	and the evolution equation.
	
	The inverse problem studied here is to determine whether the velocity field $%
	\boldsymbol{v}$ can be uniquely recovered from observations of the magnetic
	field. We represent the available measurements by an observation operator $%
	\mathcal{L}$ acting on $\boldsymbol{B}$, so that the data take the form 
	\begin{equation*}
		\mathcal{L}(\boldsymbol{B}).
	\end{equation*}
	Depending on the observation configuration, $\mathcal{L}$ may represent
	measurements over a subregion of the domain, over a restricted time
	interval, or other partial observations specified in the subsequent
	analysis. The central questions are whether the observations determine the
	velocity uniquely, which compatibility conditions are necessary for
	reconstruction, and how the geometric and boundary assumptions influence
	identifiability.
	
	\subsection*{Related work}
	
	Velocity reconstruction from electromagnetic measurements has been
	investigated through analytical, numerical, and experimental approaches. The
	reconstruction problem using magnetic-field and electric-potential
	measurements was formulated in \cite{stefani1999velocity}. Non-uniqueness of
	the reconstruction problem was subsequently analyzed in
	\cite{stefani2000uniqueness}, where a regularization based on minimizing the
	kinetic energy of the velocity field was proposed as a selection criterion.
	A contactless method for velocity reconstruction in electrically conducting
	fluids was developed in \cite{stefani2000contactless}, and theoretical and
	experimental aspects of inverse problems in magnetohydrodynamics were
	further investigated in \cite{stefani2006inverse}. A finite-element-based
	Bayesian approach to reconstructing the longitudinal velocity component in
	electromagnetic flow tomography was proposed in
	\cite{lehtikangas2016reconstruction}, but no uniqueness or stability results
	are proved there. It is worth mentioning that recently, in
	\cite{huang2022carleman}, a Carleman estimate was established for the full
	MHD system; however, that work recovers a source term that is independent of the magnetic field or the velocity of the fluid
	
	More recently, \cite{mokhtari2026velocity} established uniqueness criteria
	for the linearized induction equation in whole-space and periodic settings
	with constant background magnetic field.
	
	\subsection*{Contribution} We prove stability and uniqueness for velocity
	reconstruction in a bounded domain from partial parabolic observations of
	the magnetic field. Magnetic observations alone are not sufficient for
	uniqueness; prescribing the velocity on the inflow boundary together with
	the magnetic boundary conditions restores uniqueness under suitable
	conditions on the space-dependent background field and bounded-domain
	transport geometry. The problem has a hybrid parabolic--hyperbolic
	structure: magnetic diffusion is parabolic, while velocity transport along
	the background field is hyperbolic. This motivates a two-stage
	reconstruction and yields a hybrid stability estimate combining a Lipschitz
	contribution from the inflow trace with a weaker modulus from the parabolic
	observation. Unlike previous work in unbounded or periodic settings with
	constant background fields, we allow a space-dependent background field.
	
	\subsection*{Organization of the paper}
	This paper is organized as follows. In Section~\ref{sec:forward}, we present
	the linearized model around a background magnetic field, the existence
	result for the forward problem, and the statement of the inverse problem. In
	Section~\ref{sec:setting}, we present the geometric and functional tools
	needed to announce and prove the main results for the inverse problem. In
	Section~\ref{sec:curl-div}, we establish a well-posedness result for a
	curl--div system that is required in the proof of the main results. In
	Section~\ref{sec:main results}, we state the main results for the inverse
	problem. In Section~\ref{sec:proof-main-results}, we prove the main results,
	and in Section~\ref{sec:numerical}, we present the numerical validation of
	both the stability and uniqueness results.

	
	\section{The model}
	
	\label{sec:forward}
	
	The nonlinearity of the induction equation makes direct analysis of the
	associated inverse problem difficult. Following the linearization strategy
	commonly used in inverse problems for magnetohydrodynamics \cite%
	{stefani1999velocity,stefani2000uniqueness}, we decompose the magnetic field
	as 
	\begin{equation}
		\boldsymbol{B}(t,\boldsymbol{x}) =\boldsymbol{F}(\boldsymbol{x})+\boldsymbol{%
			\ b}(t,\boldsymbol{x}),  \label{eq:B-decomposition}
	\end{equation}
	where $\boldsymbol{F}$ is a known, time-independent background magnetic
	field and $\boldsymbol{b}$ is the magnetic perturbation.
	
	We assume that the background field satisfies 
	\begin{equation}
		\boldsymbol{F}\in \lbrack C^{2}(\overline{\Omega })]^{3},\qquad \nabla \cdot 
		\boldsymbol{F}=0,\qquad \text{in }\Omega .  \label{F-assumption}
	\end{equation}
	
	Since $\boldsymbol{B}$ and $\boldsymbol{F}$ are divergence-free, the
	perturbation is also divergence-free: 
	\begin{equation*}
		\nabla\cdot \boldsymbol{b} =\nabla\cdot \boldsymbol{B}-\nabla\cdot 
		\boldsymbol{F}=0.
	\end{equation*}
	
	Recall from the introduction that the total magnetic field satisfies the
	boundary conditions 
	\begin{equation}
		\boldsymbol{B}\cdot \boldsymbol{n}=g_{1},\qquad \boldsymbol{n}\times (\nabla
		\times \boldsymbol{B})=\boldsymbol{g}_{2},\qquad \text{on }(0,T)\times
		\partial \Omega .  \label{eq:BC-magnetic-total}
	\end{equation}
	Because the decomposition \eqref{eq:B-decomposition} separates the known
	background field from the unknown perturbation, we impose the boundary
	condition on the total field $\boldsymbol{B}$ in such a way that the
	prescribed background field $\boldsymbol{F}$ satisfies the same boundary
	data 
	\begin{equation}
		g_{1}:=\boldsymbol{F}\cdot \boldsymbol{n},\qquad \boldsymbol{g}_{2}:= 
		\boldsymbol{n}\times (\nabla \times \boldsymbol{F}).  \label{eq:g1g2-choice}
	\end{equation}
	Consequently the perturbation satisfies homogeneous absolute boundary
	conditions 
	\begin{equation*}
		\boldsymbol{b}\cdot \boldsymbol{n}=0,\qquad \boldsymbol{n}\times (\nabla
		\times \boldsymbol{b})=\boldsymbol{0},\qquad \text{on }(0,T)\times \partial
		\Omega .
	\end{equation*}
	
	This choice is natural: the background field already satisfies the
	perfectly-conducting-wall or symmetry boundary conditions, and the
	perturbation is the difference between the total field and this known
	background. Hence the boundary data of the total field are precisely those
	of the background, so the perturbation inherits boundary data. Physically,
	the linearized perturbation does not alter the normal magnetic flux or the
	tangential current density at the boundary.
	
	Substituting \eqref{eq:B-decomposition} into the induction equation, using
	that $\boldsymbol{F}$ is time-independent and satisfies \eqref{F-assumption}
	, and dropping the bilinear term $\nabla\times(\boldsymbol{v}\times 
	\boldsymbol{b})$, we obtain the linearized system 
	\begin{equation}
		\begin{cases}
			\partial_t \boldsymbol{b}-\Delta \boldsymbol{b} =\Delta \boldsymbol{F}
			+\nabla\times(\boldsymbol{v}\times\boldsymbol{F}), & \text{in }
			(0,T)\times\Omega, \\[1mm] 
			\nabla\cdot \boldsymbol{b}=0, & \text{in }(0,T)\times\Omega, \\[1mm] 
			\boldsymbol{b}\cdot\boldsymbol{n}=0,\qquad \boldsymbol{n}\times(\nabla\times 
			\boldsymbol{b})=\boldsymbol{0}, & \text{on }(0,T)\times\partial\Omega, \\%
			[1mm] 
			\boldsymbol{b}|_{t=0}=\boldsymbol{b}_0, & \text{in }\Omega,%
		\end{cases}
		\label{b-system}
	\end{equation}
	where $\boldsymbol{b}_0=\boldsymbol{B}_0-\boldsymbol{F}$.
	
	
	\subsection{The Hodge Laplacian}
	
	\label{subsec:hodge-helmholtz} 
	
	It has been noted in several works, see, for example, \cite%
	{Gerbeth2005,monniaux2016stokes,gaudin2024hodge}, that, in $\Omega \subset 
	\mathbb{R}
	^{3}$, the Laplace operator acting on vector fields, 
	\begin{equation}
		-\Delta _{\mathrm{H}}\boldsymbol{u}:=-\Delta \boldsymbol{u}=\nabla \times
		(\nabla \times \boldsymbol{u})-\nabla (\nabla \cdot \boldsymbol{u}),
		\label{eq:hodge-laplacian}
	\end{equation}
	with the boundary conditions 
	\begin{equation}
		\boldsymbol{u}\cdot \boldsymbol{n}=0,\qquad \boldsymbol{n}\times (\nabla
		\times \boldsymbol{u})=\boldsymbol{0}\qquad \text{on }\partial \Omega ,
		\label{eq:hodge-boundary}
	\end{equation}
	is called the \emph{Hodge Laplacian}.
	
	We introduce the solenoidal Hilbert space 
	\begin{equation}
		\mathbb{H}:=\left\{ \boldsymbol{u}\in \lbrack L^{2}(\Omega )]^{3}:\nabla
		\cdot \boldsymbol{u}=0\ \text{in }\Omega ,\quad \boldsymbol{u}\cdot 
		\boldsymbol{n}=0\ \text{on }\partial \Omega \right\} .  \label{eq:H}
	\end{equation}
	Let 
	\begin{equation*}
		\mathbb{P}:[L^{2}(\Omega )]^{3}\longrightarrow \mathbb{H}
	\end{equation*}
	denote the Helmholtz--Leray orthogonal projection.
	
	A key property of the Hodge realization is its compatibility with $\mathbb{P}
	$ in $%
	\mathbb{R}
	^{3}$. More precisely, for $\boldsymbol{u}$ in the domain of the Hodge
	Laplacian, 
	\begin{equation}
		\mathbb{P}\boldsymbol{u}\in D(\Delta _{\mathrm{H}}),\qquad -\mathbb{P}\Delta 
		\boldsymbol{u}=-\Delta _{\mathrm{H}}\mathbb{P}\boldsymbol{u}=\nabla \times
		(\nabla \times \boldsymbol{u}),  \label{eq:P-Hodge}
	\end{equation}
	and $\mathbb{P}\boldsymbol{u}$ satisfies the same boundary conditions, 
	\begin{equation}
		(\mathbb{P}\boldsymbol{u})\cdot \boldsymbol{n}=0,\qquad \boldsymbol{n}\times %
		\bigl(\nabla \times \mathbb{P}\boldsymbol{u}\bigr)=\boldsymbol{0}\qquad 
		\text{on }\partial \Omega .  \label{eq:P-boundary}
	\end{equation}
	Thus, the compatibility is a property of the operator realizations and their
	domains, rather than an unrestricted commutation relation $\mathbb{P}\Delta
	=\Delta \mathbb{P}$ on $[L^{2}(\Omega )]^{3}$; see \cite{gaudin2024hodge}.
	
	We consider the realization of the Hodge Laplacian on $\mathbb{H}$ with
	domain 
	\begin{equation}
		D(\Delta_{\mathrm{H}}) := \left\{ \boldsymbol{u}\in\mathbb{H}: \Delta 
		\boldsymbol{u}\in[L^2(\Omega)]^3, \quad \boldsymbol{n}\times(\nabla\times 
		\boldsymbol{u}) = \boldsymbol{0} \ \text{on }\partial\Omega \right\}.
		\label{eq:hodge-domain}
	\end{equation}
	
	Since $\nabla\cdot\boldsymbol{u}=0$ for every $\boldsymbol{u}\in D(\Delta_{ 
		\mathrm{H}})$, we have 
	\begin{equation}
		-\Delta_{\mathrm{H}}\boldsymbol{u} = -\Delta\boldsymbol{u} =
		\nabla\times(\nabla\times\boldsymbol{u}).  \label{eq:hodge-curl-curl}
	\end{equation}
	
	Hence the divergence-free constraint and the boundary conditions are
	incorporated directly into the Hodge--Laplacian realization. The evolution
	problem can therefore be formulated directly on $\mathbb{H}$, without
	introducing an additional pressure term or an explicit Helmholtz projection. 
	
	
	\subsection{Well-posedness}
	
	\label{subsec:well-posedness} 
	
	For the linearized induction system \eqref{b-system}, we set 
	\begin{equation}
		\boldsymbol{h}(t)=\Delta \boldsymbol{F}+\nabla \times (\boldsymbol{v}
		(t)\times \boldsymbol{F}),\qquad t\in (0,T).  \label{h}
	\end{equation}
	
	The Hodge--Laplacian realization $\Delta_{\mathrm{H}}$ on $\mathbb{H}$ with
	absolute boundary conditions is self-adjoint and nonnegative. Consequently, $%
	\Delta_{\mathrm{H}}$ is maximal dissipative and generates an analytic $C_0 $%
	-semigroup on $\mathbb{H}$; see, for example, \cite[Theorem 7.1]%
	{mitrea2009analyticity}. We have the following result \cite[Theorem 3.2]%
	{magenes1972non}
	
	\begin{theorem}
		\label{theorem:existence} Let $T>0$. Assume that $\Omega $ is sufficiently
		smooth and let 
		\begin{equation*}
			\boldsymbol{h}\in L^{2}(0,T;\mathbb{H}),\qquad \boldsymbol{b}_{0}\in
			D(\Delta _{\mathrm{H}}^{\frac{1}{2}}).
		\end{equation*}%
		Then the problem \eqref{b-system} admits a unique strong solution 
		\begin{equation*}
			\boldsymbol{b}\in L^{2}(\left[ 0,T\right] ;D(\Delta _{\mathrm{H}}))\cap
			H^{1}(0,T;\mathbb{H}).
		\end{equation*}
	\end{theorem}
	
	We therefore define the state space 
	\begin{equation}
		\mathbb{Y}:=\boldsymbol{b}\in L^2(\left[ 0,T\right] ;D(\Delta _{\mathrm{H}%
		}))\cap H^{1}(0,T;\mathbb{H})).  \label{Y}
	\end{equation}
	
	\begin{remark}
		By definition of the source term $\boldsymbol{h}$ in \eqref{h}, the source
		term is required to belong to $\mathbb{H}$. Since 
		\begin{equation*}
			\nabla \cdot \boldsymbol{h}(t)=\nabla \cdot \left( \Delta \boldsymbol{F}%
			+\nabla \times (\boldsymbol{v}(t)\times \boldsymbol{F})\right) =\left(
			\Delta \left( \nabla \cdot \boldsymbol{F}\right) +\nabla \cdot \left( \nabla
			\times ( \boldsymbol{v}(t)\times \boldsymbol{F})\right) \right) =0,\quad 
			\mathrm{a.e \ on}\ (0,T),
		\end{equation*}
		holds identically, the remaining condition $\boldsymbol{h}\cdot \boldsymbol{n%
		}=0$ on $\partial \Omega $ gives precisely the boundary condition 
		\begin{equation*}
			\boldsymbol{n}\cdot \nabla \times (\boldsymbol{v}(t)\times \boldsymbol{F} )=-%
			\boldsymbol{n}\cdot \Delta \oldstylenums{F}\qquad \mathrm{on}\ \partial
			\Omega \ \mathrm{for} \ \mathrm{a.e \ on}\ (0,T).
		\end{equation*}
	\end{remark}
	\subsection{Statement of the inverse problem}
	
	We can now formulate the main inverse problem. Assume that $\boldsymbol{v}$
	can be written in a separated form%
	\begin{equation*}
		\boldsymbol{v}(t,\boldsymbol{x})=\alpha (t)\boldsymbol{u}(\boldsymbol{x}),
	\end{equation*}%
	where $\alpha \in L^{2}(0,T)$ is not zero on $(0,T)$ function that is
	supposed to be known.
	
	\begin{problem}[\textbf{Inverse Problem: Recovering $\boldsymbol{v}$ from
			observations $\mathcal{L}(\boldsymbol{y})$}]
		\label{problem-inverse} Determine the spatial velocity vector field $%
		\boldsymbol{u}$ in $\Omega $ from the measurements modeled by the
		observation operator $\mathcal{L}(\boldsymbol{y})$, supplemented by
		additional boundary measurements on $\boldsymbol{u}$ as required for
		well-posedness.
	\end{problem}
	
	First, observe that uniqueness of source recovery in Problem~\ref%
	{problem-inverse} cannot hold for arbitrary $\boldsymbol{v}$ so that $%
	\boldsymbol{z}\in \mathbb{H}$ defined in \ref{h}. Indeed, the linear
	operator 
	\begin{equation}
		\boldsymbol{u}\longmapsto \nabla \times \bigl(\boldsymbol{u}\times 
		\boldsymbol{F}\bigr),  \label{1}
	\end{equation}%
	possesses a nontrivial kernel. Specifically, let $c(\boldsymbol{x})$ be a
	sufficiently smooth scalar function satisfying the transport equation $%
	\nabla \cdot \left( c\boldsymbol{F}\right) =0,$ i.e, 
	\begin{equation*}
		\nabla \cdot \left( c\boldsymbol{F}\right) =\boldsymbol{F}\cdot \nabla
		c=F_{1}\partial _{1}c+F_{2}\partial _{2}c+F_{3}\partial _{3}c=0,\qquad \text{
			in }\Omega .
	\end{equation*}%
	The existence of such a function $c$ follows from the transport equation
	theory. We have 
	\begin{equation*}
		(\boldsymbol{u}+c\boldsymbol{F})\times \boldsymbol{F}=\boldsymbol{u}\times 
		\boldsymbol{F},\qquad t\in (0,T),
	\end{equation*}%
	and consequently, 
	\begin{equation}
		\nabla \times \bigl((\boldsymbol{u}+c\boldsymbol{F})\times \boldsymbol{F}%
		\bigr)=\nabla \times (\boldsymbol{u}\times \boldsymbol{F}).  \label{parallel}
	\end{equation}%
	Thus, any component of $\boldsymbol{v}$ parallel to the background field $%
	\boldsymbol{F}$ is inherently unobservable through the induction equation,
	independently of the available state measurements $\mathcal{L}$. Additional
	information regarding the source field is therefore required to establish
	unique reconstruction.
	
	Using the vector identity 
	\begin{equation*}
		\nabla \times (\boldsymbol{A}\times \boldsymbol{B})=\boldsymbol{A}(\nabla
		\cdot \boldsymbol{B})-\boldsymbol{B}(\nabla \cdot \boldsymbol{A})+(%
		\boldsymbol{B}\cdot \nabla )\boldsymbol{A}-(\boldsymbol{A}\cdot \nabla )%
		\boldsymbol{B},
	\end{equation*}%
	with $\boldsymbol{A}=\boldsymbol{u}$, $\boldsymbol{B}=\boldsymbol{F}$, and
	taking into account $\nabla \cdot \boldsymbol{u}=0$ and $\nabla \cdot 
	\boldsymbol{F}=0$, the operator \eqref{1} expands as 
	\begin{equation}
		\boldsymbol{u}\longmapsto \nabla \times (\boldsymbol{u}\times \boldsymbol{F}%
		)=(\boldsymbol{F}\cdot \nabla )\boldsymbol{u}-(\boldsymbol{u}\cdot \nabla )%
		\boldsymbol{F},\qquad t\in (0,T).  \label{First-order}
	\end{equation}%
	To make this structural non-uniqueness explicit, recall that by the idendity %
	\ref{First-order} the operator \ref{1} is a first-order transport operator
	along the characteristics of $\boldsymbol{F}$. In particular, parallel
	perturbations of the form $c\boldsymbol{F}$ cannot be detected via the curl
	operator alone.
	
	To eliminate this ambiguity, we supplement the internal state observations
	with measurements of $\boldsymbol{u}$ restricted to the inflow boundaryin
	order to ensures unique source reconstruction.


	
	
	
	\section{Geometric and Functional settings for the inverse problem}
	\label{sec:setting}
	\subsection{Geometric setting}
	
	Let $\Gamma_{\mathrm{in}}\subset\partial\Omega$ denote the inflow boundary
	region, i.e., the portion of the boundary through which the background field 
	$\boldsymbol{F}$ enters the domain: 
	\begin{equation}
		\Gamma_{\mathrm{in}} := \left\{ \boldsymbol{x}\in\partial\Omega: \boldsymbol{%
			F}(\boldsymbol{x})\cdot\boldsymbol{n}(\boldsymbol{x})<0 \right\},
		\label{Gamma_in}
	\end{equation}
	where $\boldsymbol{n}$ denotes the outward unit normal to $\partial\Omega$.
	
	Since $\boldsymbol{F}\in[C^{2}(\overline{\Omega})]^{3}$, we extend $%
	\boldsymbol{F}$ to a $C^{2}$ vector field on $\mathbb{R}^{3}$ with compact
	support. By the Cauchy--Lipschitz theorem, the associated global flow 
	\begin{equation*}
		\boldsymbol{X} \in C^{1}(\mathbb{R}\times\mathbb{R}^{3};\mathbb{R}^{3})
	\end{equation*}
	is uniquely defined by 
	\begin{equation}
		\partial_t\boldsymbol{X}(t,\boldsymbol{x}_{0}) = \boldsymbol{F}\bigl(%
		\boldsymbol{X}(t,\boldsymbol{x}_{0})\bigr), \qquad \boldsymbol{X}(0,%
		\boldsymbol{x}_{0}) = \boldsymbol{x}_{0}.  \label{flow}
	\end{equation}
	
	\begin{definition}[$\Omega$-filling]
		\label{Definition-Omega-Filing} Let $\boldsymbol{F}$ generate the flow $(t,%
		\boldsymbol{x}_{0})\mapsto\boldsymbol{X}(t,\boldsymbol{x}_{0})$, let $%
		\Gamma_{\mathrm{in}}$ be the inflow boundary with surface measure $\sigma$,
		let $\tau(\boldsymbol{x}_{0})>0$ be the exit times, and set 
		\begin{equation*}
			\mathcal{D}:=\{(t,\boldsymbol{x}_{0}):\boldsymbol{x}_{0}\in\Gamma_{\mathrm{in%
			}},\ 0<t<\tau(\boldsymbol{x}_{0})\}.
		\end{equation*}
		The flow is \emph{$\Omega$-filling} if there exist $T_{\mathrm{fill}}>0$ and
		a null set $\mathcal{N}\subset\Omega$ such that, for a.e.\ $\boldsymbol{x}%
		\in\Omega\setminus\mathcal{N}$, there exists a pair $(\boldsymbol{x}%
		_{0},\tau)$ with 
		\begin{equation*}
			\boldsymbol{x}=\boldsymbol{X}(\tau,\boldsymbol{x}_{0}), \qquad \boldsymbol{x}%
			_{0}\in\Gamma_{\mathrm{in}},\ 0<\tau\le T_{\mathrm{fill}}, \qquad 
			\boldsymbol{X}(t,\boldsymbol{x}_{0})\in\Omega \ \ \forall\,0<t<\tau.
		\end{equation*}
		Moreover, $\tau(\boldsymbol{x}_{0})\le T_{\mathrm{fill}}$ for $\sigma$-a.e.\ 
		$\boldsymbol{x}_{0}\in\Gamma_{\mathrm{in}}$.
	\end{definition}
	
	\begin{lemma}[Injectivity of the characteristic flow]
		\label{lem:flow-injective} The flow 
		\begin{equation*}
			\boldsymbol{X}:\mathcal{D}\to\Omega,\qquad (t,\boldsymbol{x}_0)\mapsto 
			\boldsymbol{X}(t,\boldsymbol{x}_0),
		\end{equation*}
		is injective.
	\end{lemma}
	
	\begin{proof}
		Suppose
		\[
		\boldsymbol{X}(t_1,\boldsymbol{x}_1)=\boldsymbol{X}(t_2,\boldsymbol{x}_2)
		\]
		with \((t_1,\boldsymbol{x}_1),(t_2,\boldsymbol{x}_2)\in\mathcal{D}\).
		Without loss of generality, assume \(t_1\ge t_2\).
		
		If \(t_1>t_2\), set \(s=t_1-t_2>0\). By uniqueness of solutions to the autonomous ODE,
		\[
		\boldsymbol{x}_2=\boldsymbol{X}(s,\boldsymbol{x}_1).
		\]
		But
		\[
		0<s=t_1-t_2<t_1<\tau(\boldsymbol{x}_1),
		\]
		so \(\boldsymbol{X}(s,\boldsymbol{x}_1)\in\Omega\). Hence \(\boldsymbol{x}_2\in\Omega\), which contradicts \(\boldsymbol{x}_2\in\Gamma_{\mathrm{in}}\subset\partial\Omega\). Therefore \(t_1=t_2\).
		
		Now
		\[
		\boldsymbol{X}(t_1,\boldsymbol{x}_1)=\boldsymbol{X}(t_1,\boldsymbol{x}_2).
		\]
		Since \(t_1<\tau(\boldsymbol{x}_1)\) and \(t_1<\tau(\boldsymbol{x}_2)\), uniqueness of the ODE applied backward in time yields \(\boldsymbol{x}_1=\boldsymbol{x}_2\). Thus \(\boldsymbol{X}\) is injective on \(\mathcal{D}\).
	\end{proof}
	
	\begin{remark}
		\label{rem:surjectivity-condition} A standard sufficient condition for the 
		\emph{$\Omega$-filling} property is the existence of a unit vector $%
		\boldsymbol{\gamma}$ and a constant $\lambda>0$ such that \cite[Proposition~1%
		]{azerad1996inegalite} 
		\begin{equation}
			\boldsymbol{F}(\boldsymbol{x})\cdot\boldsymbol{\gamma} > \lambda \qquad\text{%
				for all }\boldsymbol{x}\in\overline{\Omega}.  \label{Filling-Condition}
		\end{equation}
		This ensures every characteristic enters through $\Gamma_{\mathrm{in}}$,
		traverses $\Omega$ in finite time, and covers the domain. By Lemma~\ref%
		{lem:flow-injective}, the representation is automatically unique, so the
		flow is a bijection from $\mathcal{D}$ onto $\Omega\setminus\mathcal{N}$.
	\end{remark}
	
	\begin{remark}
		In the constant case $\boldsymbol{F}(\boldsymbol{x})\equiv \boldsymbol{F}%
		_{0}\neq 0$, the flow is explicit: 
		\begin{equation*}
			\boldsymbol{X}(t,\boldsymbol{x}_{0})=\boldsymbol{x}_{0}+t\boldsymbol{F}_{0}.
		\end{equation*}
		Condition~\eqref{Filling-Condition} holds trivially with $\boldsymbol{\gamma}%
		=\|\boldsymbol{F}_{0}\|^{-1}\boldsymbol{F}_{0}$, and injectivity follows
		from Lemma~\ref{lem:flow-injective}. If the $\Omega$-filling property also
		holds, then $\boldsymbol{X}$ is a bijection from $\mathcal{D}$ onto $%
		\Omega\setminus\mathcal{N}$.
	\end{remark}
	
	\subsection{Functional setting}
	
	\begin{definition}[The space $\mathbb{X}_{\boldsymbol{F}}$]
		\label{Definition-XF} We define 
		\begin{equation}
			\mathbb{X}_{\boldsymbol{F}}:=\left\{ \boldsymbol{u}\in \left[ L^{2}(\Omega )%
			\right] ^{3}:\nabla \cdot \boldsymbol{u}=0,\quad \nabla \times (\boldsymbol{u%
			}\times \boldsymbol{F})\in \mathbb{H}\right\} ,  \label{X_F}
		\end{equation}%
		equipped with the norm 
		\begin{equation*}
			\Vert \boldsymbol{u}\Vert _{\mathbb{X}_{\boldsymbol{F}}}^{2}:=\left\Vert 
			\boldsymbol{u}\right\Vert _{\mathbb{H}}^{2}+\left\Vert \nabla \times (%
			\boldsymbol{u}\times \boldsymbol{F})\right\Vert _{\mathbb{H}}^{2}.
		\end{equation*}
	\end{definition}
	
	Under the \emph{$\Omega$-filling} assumption, Proposition 2.4 of \cite%
	{brunken2019parametrized} yields a bounded trace operator 
	\begin{equation*}
		\gamma _{-}:\mathbb{X}_{\boldsymbol{F}}\longrightarrow L^{2}\bigl(\Gamma _{%
			\mathrm{in}};\tau|\boldsymbol{F}\cdot \boldsymbol{n}|\bigr),
	\end{equation*}
	where 
	\begin{equation*}
		L^{2}\bigl(\Gamma _{\mathrm{in}};\tau|\boldsymbol{F}\cdot \boldsymbol{n}|%
		\bigr):=\left\{ \boldsymbol{g}:\Gamma _{\mathrm{in}}\rightarrow \mathbb{R}%
		^{3}:\boldsymbol{g}\ \text{is measurable and}\ \int_{\Gamma _{\mathrm{in}}}|%
		\boldsymbol{g}|^{2}\tau|\boldsymbol{F}\cdot \boldsymbol{n}|\,d\sigma <\infty
		\right\},
	\end{equation*}
	equipped with the norm 
	\begin{equation*}
		\Vert \boldsymbol{g}\Vert _{L^{2}(\Gamma _{\mathrm{in}};\tau|\boldsymbol{F}%
			\cdot \boldsymbol{n}|)}^{2}:=\int_{\Gamma _{\mathrm{in}}}|\boldsymbol{g}%
		|^{2}\tau|\boldsymbol{F}\cdot \boldsymbol{n}|\,d\sigma .
	\end{equation*}
	We have:
	
	\begin{proposition}
		\label{proposition-trace} Assume that $\boldsymbol{F}$ satisfies %
		\eqref{F-assumption}, and that the flow $\boldsymbol{X}$ is \emph{$\Omega$-filling} in the sense of Definition~\ref{Definition-Omega-Filing}. Then
		there exists a positive constant $C>0$ such that 
		\begin{equation}
			\int_{\Gamma _{\mathrm{in}}}|\gamma _{-}\boldsymbol{u}|^{2}\,\tau \,|%
			\boldsymbol{F}\cdot \boldsymbol{n}|\,\mathrm{d}\sigma \leq C\,\Vert 
			\boldsymbol{u}\Vert _{\mathbb{X}_{\boldsymbol{F}}}^{2},\qquad \boldsymbol{u}%
			\in \mathbb{X}_{\boldsymbol{F}}.  \label{trace-estimate}
		\end{equation}
	\end{proposition}
	
	\begin{proof}
		Fix $\boldsymbol{u}\in\mathbb{X}_{\boldsymbol{F}}$ and set
		\[
		\boldsymbol{w}
		:=
		\nabla\times(\boldsymbol{u}\times\boldsymbol{F}),
		\qquad
		\boldsymbol{f}
		:=
		\boldsymbol{w}+(\boldsymbol{u}\cdot\nabla)\boldsymbol{F}.
		\]
		The vector identity
		\[
		\nabla\times(\boldsymbol{u}\times\boldsymbol{F})
		=
		\boldsymbol{u}\,\nabla\cdot\boldsymbol{F}
		-\boldsymbol{F}\,\nabla\cdot\boldsymbol{u}
		+(\boldsymbol{F}\cdot\nabla)\boldsymbol{u}
		-(\boldsymbol{u}\cdot\nabla)\boldsymbol{F}
		\]
		and the divergence-free assumptions imply, in distributions, that
		\[
		(\boldsymbol{F}\cdot\nabla)\boldsymbol{u}
		=
		\boldsymbol{f}.
		\]
		Moreover,
		\[
		\|\boldsymbol{f}\|_{L^2(\Omega)}
		\le
		\|\boldsymbol{w}\|_{L^2(\Omega)}
		+
		\|D\boldsymbol{F}\|_{L^\infty(\Omega)}
		\|\boldsymbol{u}\|_{L^2(\Omega)}.
		\]
		
		The change-of-variables formula transforms this distributional identity
		into
		\[
		\partial_t(\boldsymbol{u}\circ\boldsymbol{X})
		=
		\boldsymbol{f}\circ\boldsymbol{X}
		\]
		in flow coordinates. This can be verified by testing locally in
		flow-coordinate charts. Since both pullbacks are square-integrable in
		the flow-coordinate measure, Fubini's theorem implies that
		$t\mapsto\boldsymbol{u}(\boldsymbol{X}(t,\boldsymbol{x}_0))$
		belongs to $H^1(0,\tau(\boldsymbol{x}_0))$ for almost every
		$\boldsymbol{x}_0\in\Gamma_{\mathrm{in}}$.
		
		For $\boldsymbol{z}\in H^1(0,\tau;\mathbb{R}^d)$,
		\[
		\tau|\boldsymbol{z}(0)|^2
		\le
		2\int_0^\tau|\boldsymbol{z}(t)|^2\,\mathrm{d}t
		+
		\tau^2\int_0^\tau|\boldsymbol{z}'(t)|^2\,\mathrm{d}t.
		\]
		Apply this inequality along almost every flow line, multiply by
		$|\boldsymbol{F}(\boldsymbol{x}_0)\cdot
		\boldsymbol{n}(\boldsymbol{x}_0)|$, and integrate over
		$\Gamma_{\mathrm{in}}$. The change-of-variables formula and
		$\tau\le T_{\mathrm{fill}}$ give
		\[
		\int_{\Gamma_{\mathrm{in}}}
		|\gamma_{-}\boldsymbol{u}|^{2}\,
		\tau|\boldsymbol{F}\cdot\boldsymbol{n}|\,\mathrm{d}\sigma
		\le
		2\|\boldsymbol{u}\|_{L^2(\Omega)}^{2}
		+
		T_{\mathrm{fill}}^{2}\|\boldsymbol{f}\|_{L^2(\Omega)}^{2}.
		\]
		Using the bound for $\boldsymbol{f}$ and
		$(a+b)^2\le 2a^2+2b^2$, we obtain
		\[
		\int_{\Gamma_{\mathrm{in}}}
		|\gamma_{-}\boldsymbol{u}|^{2}\,
		\tau|\boldsymbol{F}\cdot\boldsymbol{n}|\,\mathrm{d}\sigma
		\le
		C\left(
		\|\boldsymbol{u}\|_{L^2(\Omega)}^{2}
		+
		\|\boldsymbol{w}\|_{L^2(\Omega)}^{2}
		\right),
		\]
		where one may take
		\[
		C=
		\max\left\{
		2+2T_{\mathrm{fill}}^{2}\|D\boldsymbol{F}\|_{L^\infty}^{2},
		\;2T_{\mathrm{fill}}^{2}
		\right\},		\]
		which is exactly the estimate \eqref{trace-estimate}.
	\end{proof}
	
	\begin{remark}[$L^2_{\mathrm{loc}}$-trace]
		\label{remark:tau-essential} In previous works, the trace of $\boldsymbol{u}%
		\in\mathbb{X}_{\boldsymbol{F}}$ on the inflow boundary $\Gamma_{\mathrm{in}} 
		$ is only known to be bounded in $L^2_{\mathrm{loc}}(\Gamma_{\mathrm{in}};|%
		\boldsymbol{F}\cdot\boldsymbol{n}|\,\mathrm{d}\sigma)$, without the factor $%
		\tau$ (see for instance the survey \cite{imagawa2025revisit}). This local
		regularity is not sufficient to obtain a global estimate in $L^2(\Gamma_{%
			\mathrm{in}};|\boldsymbol{F}\cdot\boldsymbol{n}|\,\mathrm{d}\sigma)$,
		because the exit time $\tau(\boldsymbol{x}_0)$ degenerates near the grazing
		part of $\Gamma_{\mathrm{in}}$, where $\boldsymbol{F}\cdot\boldsymbol{n}\to0$%
		. The weight $\tau$ in \eqref{trace-estimate} is therefore essential: it
		compensates for this degeneracy and yields a global trace estimate in $%
		L^2(\Gamma_{\mathrm{in}};\tau|\boldsymbol{F}\cdot\boldsymbol{n}|\,\mathrm{d}%
		\sigma)$ uniformly over $\mathbb{X}_{\boldsymbol{F}}$.
	\end{remark}
	
	\begin{remark}[Sharpness of the weight]
		\label{rem:sharpness-tau-weight} The factor $\tau$ in the trace estimate
		cannot in general be omitted. More precisely, for $\Omega=B(0,1)\subset%
		\mathbb{R}^{3}$ and $\boldsymbol{F}=\boldsymbol{e}_{1}$, there is no
		constant $C$ such that * 
		\begin{equation*}
			\|\gamma_-\boldsymbol{u}\|_{L^{2} (\Gamma_{\mathrm{in}};\,|\boldsymbol{F}%
				\cdot\boldsymbol{n}|\,\mathrm{d}\sigma)} \le C\|\boldsymbol{u}\|_{\mathbb{X}%
				_{\boldsymbol{F}}} \qquad \text{for all }\boldsymbol{u}\in\mathbb{X}_{%
				\boldsymbol{F}}.
		\end{equation*}
		In fact, for every $\alpha<1$, the same is true if the boundary measure is
		weighted by $\tau^\alpha$ instead of $\tau$.
		
		To see this, note that 
		\begin{equation*}
			\Gamma_{\mathrm{in}} = \{\boldsymbol{y}\in\partial B(0,1):y_1<0\}, \qquad 
			\boldsymbol{X}(t,\boldsymbol{y}) =\boldsymbol{y}+t\boldsymbol{e}_1, \qquad
			\tau(\boldsymbol{y})=-2y_1.
		\end{equation*}
		Write $r^2=y_2^2+y_3^2$. On the inflow hemisphere, 
		\begin{equation*}
			y_1=-\sqrt{1-r^2}, \qquad \tau=2\sqrt{1-r^2}, \qquad |\boldsymbol{F}\cdot%
			\boldsymbol{n}|\,\mathrm{d}\sigma =\mathrm{d}y_2\,\mathrm{d}y_3.
		\end{equation*}
		
		Fix $0<\varepsilon<1$, and choose a nonzero $\psi_\varepsilon\in C_c^\infty(%
		\mathbb{R}^{2})$ whose support is contained in the annulus 
		\begin{equation*}
			A_\varepsilon := \left\{(y_2,y_3): 1-\left(\frac{\varepsilon}{2}\right)^2
			<y_2^2+y_3^2<1 \right\}.
		\end{equation*}
		Define 
		\begin{equation*}
			\boldsymbol{u}_\varepsilon(\boldsymbol{x}) := \nabla\times \bigl(%
			\psi_\varepsilon(x_2,x_3)\boldsymbol{e}_1\bigr) = \bigl(0,\partial_3\psi_%
			\varepsilon, -\partial_2\psi_\varepsilon\bigr).
		\end{equation*}
		Then 
		\begin{equation*}
			\nabla\cdot\boldsymbol{u}_\varepsilon=0, \qquad \nabla\times(\boldsymbol{u}%
			_\varepsilon\times\boldsymbol{e}_1)=0.
		\end{equation*}
		Thus $\boldsymbol{u}_\varepsilon\in\mathbb{X}_{\boldsymbol{F}}$ and 
		\begin{equation*}
			\|\boldsymbol{u}_\varepsilon\|_{\mathbb{X}_{\boldsymbol{F}}}^{2} = \|%
			\boldsymbol{u}_\varepsilon\|_{L^{2}(\Omega)}^{2}.
		\end{equation*}
		The field is constant along each flow line, so the flow change-of-variables
		formula gives 
		\begin{equation*}
			\|\boldsymbol{u}_\varepsilon\|_{L^{2}(\Omega)}^{2} = \int_{A_\varepsilon}
			\tau(y_2,y_3) |\nabla\psi_\varepsilon(y_2,y_3)|^{2} \,\mathrm{d}y_2\,\mathrm{%
				d}y_3.
		\end{equation*}
		On the other hand, 
		\begin{equation*}
			\|\gamma_-\boldsymbol{u}_\varepsilon\|_ {L^{2}(\Gamma_{\mathrm{in}}; \,|%
				\boldsymbol{F}\cdot\boldsymbol{n}|\,\mathrm{d}\sigma)}^{2} =
			\int_{A_\varepsilon} |\nabla\psi_\varepsilon|^{2} \,\mathrm{d}y_2\,\mathrm{d}%
			y_3.
		\end{equation*}
		Since $\tau<\varepsilon$ on $A_\varepsilon$, 
		\begin{equation*}
			\|\boldsymbol{u}_\varepsilon\|_{\mathbb{X}_{\boldsymbol{F}}}^{2} \le
			\varepsilon \int_{A_\varepsilon}|\nabla\psi_\varepsilon|^{2} \,\mathrm{d}%
			y_2\,\mathrm{d}y_3.
		\end{equation*}
		Consequently, 
		\begin{equation*}
			\frac{ \|\gamma_-\boldsymbol{u}_\varepsilon\|_ {L^{2}(\Gamma_{\mathrm{in}};
					\,|\boldsymbol{F}\cdot\boldsymbol{n}|\,\mathrm{d}\sigma)} }{ \|\boldsymbol{u}%
				_\varepsilon\|_{\mathbb{X}_{\boldsymbol{F}}} } \ge \varepsilon^{-1/2},
		\end{equation*}
		so the unweighted estimate fails as $\varepsilon\rightarrow0$.
		
		More generally, for $\alpha<1$, 
		\begin{equation*}
			\int_{A_\varepsilon}\tau^\alpha|\nabla\psi_\varepsilon|^2 \,\mathrm{d}y_2\,%
			\mathrm{d}y_3 \ge \varepsilon^{\alpha-1}
			\int_{A_\varepsilon}\tau|\nabla\psi_\varepsilon|^2 \,\mathrm{d}y_2\,\mathrm{d%
			}y_3,
		\end{equation*}
		because $\tau<\varepsilon$ and $\alpha-1<0$. Thus the trace estimate also
		fails with the weight $\tau^\alpha$ for every $\alpha<1$.
	\end{remark}
	
	Since $\boldsymbol{F}\in \lbrack C^{2}(\overline{\Omega })]^{d}$ and $\Gamma
	_{\mathrm{in}}\subset \partial \Omega $ is compact, the function $%
	\boldsymbol{x}\mapsto |\boldsymbol{F}(\boldsymbol{x})\cdot \boldsymbol{n}(%
	\boldsymbol{x})|$ is continuous on $\Gamma _{\mathrm{in}}$, hence bounded:
	there exists $M>0$ such that 
	\begin{equation*}
		0\leq \tau |\boldsymbol{F}(\boldsymbol{x})\cdot \boldsymbol{n}(\boldsymbol{x}%
		)|\leq T_{\mathrm{fill}}M\qquad \text{for all }\boldsymbol{x}\in \Gamma _{%
			\mathrm{in}}.
	\end{equation*}%
	Therefore, for any $\boldsymbol{g}\in \lbrack L^{2}(\Gamma _{\mathrm{in}%
	})]^{d}$, 
	\begin{equation*}
		\Vert \boldsymbol{g}\Vert _{\lbrack L^{2}(\Gamma _{\mathrm{in}};\,\tau |%
			\boldsymbol{F}\cdot \boldsymbol{n}|)]^{d}}^{2}=\int_{\Gamma _{\mathrm{in}}}|%
		\boldsymbol{g}|^{2}\,\tau |\boldsymbol{F}\cdot \boldsymbol{n}|\,d\sigma \leq
		T_{\mathrm{fill}}M\int_{\Gamma _{\mathrm{in}}}|\boldsymbol{g}|^{2}\,d\sigma
		=T_{\mathrm{fill}}M\,\Vert \boldsymbol{g}\Vert _{\lbrack L^{2}(\Gamma _{%
				\mathrm{in}})]^{d}}^{2}.
	\end{equation*}%
	Thus, the injection of $[L^{2}(\Gamma _{\mathrm{in}})]^{d}$ in $%
	[L^{2}(\Gamma _{\mathrm{in}};\,\tau |\boldsymbol{F}\cdot \boldsymbol{n}%
	|)]^{d}$ is continuous, and since the weight $\tau |\boldsymbol{F}(%
	\boldsymbol{x})\cdot \boldsymbol{n}(\boldsymbol{x})|$ is bounded, and
	positive, the space $\left[ C^{\infty }\left( \Gamma _{\mathrm{in}}\right) %
	\right] ^{d}$ is dense in $[L^{2}(\Gamma _{\mathrm{in}};\,\tau |\boldsymbol{F%
	}\cdot \boldsymbol{n}|)]^{d}$ \cite[Proposition A.3]{duarte2024introduction}.
	
	In particular, since $H^{1}(\Gamma _{\mathrm{in}})\hookrightarrow
	L^{2}(\Gamma _{\mathrm{in}})$, we also have 
	\begin{equation*}
		\lbrack H^{1}(\Gamma _{\mathrm{in}})]^{d}\hookrightarrow \lbrack
		L^{2}(\Gamma _{\mathrm{in}};\,\tau |\boldsymbol{F}\cdot \boldsymbol{n}%
		|)]^{d}.
	\end{equation*}%
	Thus every $\boldsymbol{g}\in \lbrack H^{1}(\Gamma _{\mathrm{in}})]^{d}$
	belongs to the weighted $L^{2}$ space with weight $\tau |\boldsymbol{F}\cdot 
	\boldsymbol{n}|$.
	
	
	\section{The curl--div system}\label{sec:curl-div}
	
	In this section, we establish a well-posedness result for the stationary
	curl--div system 
	\begin{equation}
		\begin{cases}
			\nabla \times (\boldsymbol{u}\times \boldsymbol{F})=\boldsymbol{z}, & \text{
				in }\Omega , \\[1mm] 
			\nabla \cdot \boldsymbol{u}=0, & \text{in }\Omega , \\[1mm] 
			\boldsymbol{u}=\boldsymbol{g}, & \text{on }\Gamma_{\mathrm{in}} .%
		\end{cases}
		\label{eq:div-curl}
	\end{equation}
	This system is overdetermined: it consists of four scalar equations for the
	three components of $\boldsymbol{u}$. Hence solvability requires a
	compatibility condition on the data $\boldsymbol{g}$.
	
	For $\boldsymbol{g}\in [H^{1}(\Gamma_{\mathrm{in}})]^{d}$, we introduce the
	compatibility functional 
	\begin{equation}  \label{C(g,h)}
		{\mathcal{C}}(\boldsymbol{g}) := (\boldsymbol{F}\cdot\boldsymbol{n}%
		)\nabla_{\Gamma_{\mathrm{in}}}\cdot\boldsymbol{g} +\bigl( D\boldsymbol{F}\,%
		\boldsymbol{g} - D_{\Gamma}\boldsymbol{g}\,\boldsymbol{F}_{\Gamma} \bigr)%
		\cdot\boldsymbol{n} \qquad \mathrm{on}\ \Gamma_{\mathrm{in}},
	\end{equation}
	where 
	\begin{equation*}
		\boldsymbol{F}_{\Gamma}:=\boldsymbol{F}-(\boldsymbol{F}\cdot \boldsymbol{n}
		)\,\boldsymbol{n}
	\end{equation*}
	denotes the tangential component of $\boldsymbol{F}$ on $\Gamma_{\mathrm{in}
	} $, $\nabla_{\Gamma_{\mathrm{in}}}\cdot$ is the surface divergence on $%
	\Gamma_{\mathrm{in}}$, $D_{\Gamma}\boldsymbol{g}$ is the tangential Jacobian
	of $\boldsymbol{g}$.
	
	\begin{proposition}
		\label{Proposition2} Assume that $\Omega$ is sufficiently smooth and that $%
		\boldsymbol{F}$ satisfies \ref{F}. Let $(\boldsymbol{g},\boldsymbol{z})\in
		[H^{1}(\Gamma_{\mathrm{in}})]^{d}\times \mathbb{H}$. Assume that the flow $%
		\boldsymbol{X}$ is \emph{$\Omega$-filling and one-to-one} in the sense of
		Definition~\ref{Definition-Omega-Filing}. Then system \eqref{eq:div-curl}
		admits a unique solution $\boldsymbol{u}$ if and only if 
		\begin{equation}  \label{compat}
			\mathcal{C}(\boldsymbol{g})=0 \qquad \text{on }\Gamma_{\mathrm{in}}.
		\end{equation}
		
		Furthermore, the solution satisfies the stability estimate 
		\begin{equation}
			\Vert \boldsymbol{u}\Vert_{\mathbb{X}_{\boldsymbol{F}}}^{2} \leq C\left(
			\Vert \boldsymbol{z}\Vert_{\mathbb{H}}^{2} + \Vert \boldsymbol{g}
			\Vert_{[H^{1}(\Gamma_{\mathrm{in}})]^{d}}^{2} \right),  \label{eq:estimate}
		\end{equation}
		where $C$ depends only on $\Omega$ and $\boldsymbol{F}$.
	\end{proposition}
	
	\begin{proof}
		See Appendix~\ref{AppendixB: proof curl-div}.
	\end{proof}
	
	\begin{remark}
		System \eqref{eq:div-curl} is a first-order transport system and can be
		recast in the form \eqref{First-order}. Consequently, a natural route to
		well-posedness is through the characteristic flow associated with $%
		\boldsymbol{F}$. In particular, in the divergence-free case $\nabla\cdot 
		\boldsymbol{F}=0$, a coercivity argument based solely on the transport
		bilinear form is not available.
		
		To make this precise, define 
		\begin{equation*}
			\Gamma_{\mathrm{out}} :=\{x\in\partial\Omega:\boldsymbol{F}(x)\cdot 
			\boldsymbol{n}(x)>0\},
		\end{equation*}
		Temporarily ignore the constraint $\nabla\cdot\boldsymbol{u}=0$. Define 
		\begin{equation*}
			\mathbb{W}:=\left\{\boldsymbol{u}\in [H^{1}(\Omega)]^3 \,:\, \boldsymbol{u}%
			=0\ \mathrm{on }\ \Gamma_{\mathrm{in}}\right\}
		\end{equation*}
		and consider the transport bilinear form 
		\begin{equation*}
			a(\boldsymbol{u},\boldsymbol{w}) :=\int_{\Omega}\Big( (\boldsymbol{F}
			\cdot\nabla)\boldsymbol{u} -(D\boldsymbol{F})\boldsymbol{u}\Big)\cdot 
			\boldsymbol{w}\,\mathrm{d}\boldsymbol{x},
		\end{equation*}
		where $(D\boldsymbol{F})_{ij}=\partial_{j}F_{i}$. Taking $\boldsymbol{w}= 
		\boldsymbol{u}$ and using $\nabla\cdot\boldsymbol{F}=0$, integration by
		parts yields 
		\begin{equation*}
			a(\boldsymbol{u},\boldsymbol{u}) =\frac12\int_{\partial\Omega} ( \boldsymbol{%
				F}\cdot\boldsymbol{n})|\boldsymbol{u}|^{2}\,\mathrm{d}S -\int_{\Omega}%
			\boldsymbol{u}^{T}(D\boldsymbol{F})\boldsymbol{u} \,\mathrm{d} \boldsymbol{x}%
			.
		\end{equation*}
		Since $\boldsymbol{u}=0$ on $\Gamma_{\mathrm{in}}$, the boundary term
		reduces to 
		\begin{equation*}
			a(\boldsymbol{u},\boldsymbol{u}) =\frac12\int_{\Gamma_{\mathrm{out}}} ( 
			\boldsymbol{F}\cdot\boldsymbol{n})|\boldsymbol{u}|^{2}\,\mathrm{d}S
			-\int_{\Omega}\boldsymbol{u}^{T}(D\boldsymbol{F})\boldsymbol{u} \,\mathrm{d} 
			\boldsymbol{x}.
		\end{equation*}
		On $\Gamma_{\mathrm{out}}$, we have $\boldsymbol{F}\cdot\boldsymbol{n}>0$,
		so the boundary contribution is nonnegative. Hence the volume term alone
		cannot be coercive: a pointwise negativity condition of the form 
		\begin{equation*}
			\boldsymbol{\xi}^{T}(D\boldsymbol{F})(\boldsymbol{x})\,\boldsymbol{\xi} \le
			-c_{0}|\boldsymbol{\xi}|^{2} \qquad\text{for all }\boldsymbol{\xi}\in 
			\mathbb{R}^{d} \ \text{and a.e. }\boldsymbol{x}\in\Omega
		\end{equation*}
		would be sufficient, but it is incompatible with $\nabla\cdot\boldsymbol{F}
		=0$. Indeed, summing the inequality for $\boldsymbol{\xi}=e_{i}$ over the
		canonical basis $\{e_{1},\dots,e_{d}\}$ gives 
		\begin{equation*}
			\nabla\cdot\boldsymbol{F} =\mathrm{trace}(D\boldsymbol{F}) =\sum_{i=1}^{d}
			e_{i}^{T}(D\boldsymbol{F})e_{i} \le -dc_{0}<0,
		\end{equation*}
		which contradicts $\nabla\cdot\boldsymbol{F}=0$.
		
		Therefore, the stationary curl--div system \eqref{eq:div-curl} is properly
		analyzed by propagating the inflow data along the characteristics of $%
		\boldsymbol{F}$.
	\end{remark}
	
	\begin{remark}
		The curl-div system \eqref{eq:div-curl} is a particular case of the general
		system 
		\begin{equation*}
			\begin{cases}
				\nabla \times (\boldsymbol{A}\boldsymbol{u})=\boldsymbol{z}, & \mathrm{in}\
				\Omega , \\[1mm] 
				\nabla \cdot \boldsymbol{u}=0, & \mathrm{in}\ \Omega , \\[1mm] 
				\text{suitable boundary conditions}, & \mathrm{on}\ \Gamma _{\mathrm{in}},%
			\end{cases}%
		\end{equation*}
		which has been studied under the assumption that $\boldsymbol{A}$ is
		uniformly positive definite; see, for example, \cite{alonso2019curl}. In the
		present setting, 
		\begin{equation*}
			\boldsymbol{A}\boldsymbol{u}=\boldsymbol{u}\times \boldsymbol{F},
		\end{equation*}
		and hence $\boldsymbol{A}$ is singular, since $\boldsymbol{u}\times 
		\boldsymbol{F}=\boldsymbol{0}$ whenever $\boldsymbol{u}$ is parallel to $%
		\boldsymbol{F}$. Consequently, the results of \cite{alonso2019curl} do not
		apply directly to the present degenerate setting. A systematic
		characterization of well-posedness for such singular curl-div systems is
		beyond the scope of this work.
	\end{remark}
	
	\section{Main results}\label{sec:main results}
	
	We now establish a global stability estimate for the inverse problem,
	showing that the unknown source $\boldsymbol{u}$ depends continuously on the
	observation data and its prescribed boundary trace.
	
	To encompass different measurement configurations, we introduce a bounded
	linear observation operator 
	\begin{equation}
		\mathcal{L}:\mathbb{Y}\longrightarrow\mathbb{H}_{\mathrm{obs}},  \label{F}
	\end{equation}
	where $\mathbb{Y}$ is the state space defined in~\eqref{Y} and $\mathbb{H}_{ 
		\mathrm{obs}}$ is the corresponding observation space. The results below
	depend on $\mathcal{L}$ only through the injectivity or stability property
	assumed for the associated parabolic inverse source problem.
	
	The framework includes, in particular, internal observations, boundary flux
	observations, and final-time measurements, such as 
	\begin{equation}
		\mathcal{L}(\boldsymbol{b}) = \boldsymbol{b}|_{(0,T)\times\omega_{\mathrm{\
					obs }}}, \qquad \mathbb{H}_{\mathrm{obs}} = L^2(0,T;[L^2(\omega_{\mathrm{obs}
		})]^d),  \label{internal-F}
	\end{equation}
	or
	
	\begin{equation}
		\mathcal{L}(\boldsymbol{b}) = \boldsymbol{b}(T,\cdot)|_{\omega_{\mathrm{obs}
		}}, \qquad \mathbb{H}_{\mathrm{obs}} = [L^2(\omega_{\mathrm{obs}})]^d.
		\label{Final-F}
	\end{equation}
	Such observation configurations are standard in parabolic inverse source
	problems; see, e.g., \cite%
	{KianSambouSoccorsi2020,Yamamoto2009,SakamotoYamamoto2009}.
	
	Accordingly, for system~\eqref{b-system}, we define the generalized
	observation operator 
	\begin{equation}
		\begin{aligned} \Lambda_{\mathcal{L}}:\mathbb{X}_{\boldsymbol{F}}
			&\longrightarrow
			\mathbb{H}_{\mathrm{obs}}\times[H^1(\Gamma_{\mathrm{in}})]^d,\\
			\boldsymbol{u} &\longmapsto \bigl(\mathcal{L}(\boldsymbol{b}),
			\boldsymbol{u}|_{\Gamma_{\mathrm{in}}}\bigr), \end{aligned}  \label{L_gen}
	\end{equation}
	where $\boldsymbol{b}$ is the state generated by $\boldsymbol{v}=\alpha 
	\boldsymbol{u}$ through system~\eqref{b-system}.
	
	\subsection{Stability result}
	
	We first introduce the auxiliary parabolic system for the difference fields.
	For $i=1,2$, let $\boldsymbol{z}_i\in\mathbb{H}$ and let $\boldsymbol{y}_i$
	solve 
	\begin{equation}
		\begin{cases}
			\partial_{t}\boldsymbol{y}_i-\Delta\boldsymbol{y}_i =\alpha(t)\boldsymbol{z}%
			_i+\Delta\boldsymbol{F}, & \text{in }(0,T)\times\Omega, \\[1mm] 
			\nabla\cdot\boldsymbol{y}_i=0, & \text{in }(0,T)\times\Omega, \\[1mm] 
			\boldsymbol{y}_i\cdot\boldsymbol{n}=0,\qquad \boldsymbol{n}%
			\times(\nabla\times\boldsymbol{y}_i)=\boldsymbol{0}, & \text{on }%
			(0,T)\times\partial\Omega, \\[1mm] 
			\boldsymbol{y}_i|_{t=0}=\boldsymbol{y}_{0}, & \text{in }\Omega.%
		\end{cases}
		\label{eq:abstract-inverseh}
	\end{equation}
	By Theorem~\ref{theorem:existence}, for $\alpha\boldsymbol{z}_i+\Delta%
	\boldsymbol{F}\in L^{2}(0,T;\mathbb{H})$, the unique solution to %
	\eqref{eq:abstract-inverseh} belongs to $\mathbb{Y}$ given in~\eqref{Y}. The
	difference $\boldsymbol{y}:=\boldsymbol{y}_1-\boldsymbol{y}_2$ solves 
	\begin{equation}
		\begin{cases}
			\partial_{t}\boldsymbol{y}-\Delta \boldsymbol{y} =\alpha(t)(\boldsymbol{z}_1-%
			\boldsymbol{z}_2), & \text{in }(0,T)\times\Omega, \\[1mm] 
			\nabla\cdot\boldsymbol{y}=0, & \text{in }(0,T)\times\Omega, \\[1mm] 
			\boldsymbol{y}\cdot\boldsymbol{n}=0,\qquad \boldsymbol{n}\times(\nabla\times%
			\boldsymbol{y})=\boldsymbol{0}, & \text{on }(0,T)\times\partial\Omega, \\%
			[1mm] 
			\boldsymbol{y}|_{t=0}=\boldsymbol{0}, & \text{in }\Omega.%
		\end{cases}
		\label{eq:abstract-inverseh-diff}
	\end{equation}
	For the magnetic perturbations $\boldsymbol{b}_i$ generated by $\boldsymbol{v%
	}_i=\alpha\boldsymbol{u}_i$ from the same initial data $\boldsymbol{b}_0$
	and the same background field $\boldsymbol{F}$, one has 
	\begin{equation}
		\boldsymbol{b}_1-\boldsymbol{b}_2 =\boldsymbol{y}_1-\boldsymbol{y}_2,
		\label{eq:b-y-difference}
	\end{equation}
	
	\begin{theorem}
		\label{main-stability}  Assume that $\boldsymbol{F}$ satisfies~%
		\eqref{F-assumption} and let  $\gamma_{-}\boldsymbol{u}_i\in[H^{1}(\Gamma_{%
			\mathrm{in}})]^{3}$  (in the trace sense) be given for $i=1,2$. Let  $%
		\Theta:[0,\infty)\to[0,\infty)$ be a continuous function with  $%
		\lim_{s\to0}\Theta(s)=0$. Assume that:
		
		\begin{enumerate}
			\item[$\mathrm{(i)}$] $\alpha\in L^{2}(0,T)$ is such that the auxiliary 
			parabolic inverse source problem~\eqref{eq:abstract-inverseh-diff} admits 
			the stability estimate  
			\begin{equation}
				\|\boldsymbol{z}_1-\boldsymbol{z}_2\|_{\mathbb{H}}^{2} \leq \Theta\left( \|%
				\mathcal{L}(\boldsymbol{y}_1-\boldsymbol{y}_2)\|_ {\mathbb{H}_{\mathrm{obs}%
				}}^{2} \right),  \label{eq:parabolic-stab}
			\end{equation}
			where $\boldsymbol{y}_i$ is the solution of  \eqref{eq:abstract-inverseh}
			associated with $\boldsymbol{z}_i$ for  $i=1,2$, and $\mathcal{L}$ is the
			observation operator defined  in~\eqref{F}.
			
			\item[$\mathrm{(ii)}$] The boundary measurement difference  $\gamma_{-}(%
			\boldsymbol{u}_1-\boldsymbol{u}_2)$ satisfies the  compatibility condition  
			\begin{equation}
				\mathcal{C}\left( \gamma_{-}(\boldsymbol{u}_1-\boldsymbol{u}_2) \right)=0,
				\qquad\text{on }\Gamma_{\mathrm{in}},  \label{C(h,g)=0}
			\end{equation}
			where $\mathcal{C}$ is defined in~\eqref{C(g,h)}.
			
			\item[$\mathrm{(iii)}$] The background vector field $\boldsymbol{F}$ 
			satisfies~\eqref{F-assumption} and the flow $\boldsymbol{X}$ is  \emph{$%
				\Omega$-filling and one-to-one} in the sense of  Definition~\ref%
			{Definition-Omega-Filing}. 
		\end{enumerate}
		
		Then the velocity difference  $\boldsymbol{u}_1-\boldsymbol{u}_2\in\mathbb{X}%
		_{\boldsymbol{F}}$  satisfies the stability estimate  
		\begin{equation}
			\begin{aligned} \|\boldsymbol{u}_1-\boldsymbol{u}_2\|_
				{\mathbb{X}_{\boldsymbol{F}}}^{2} &\leq C\Bigl[ \Theta\left(
				\|\mathcal{L}(\boldsymbol{y}_1-\boldsymbol{y}_2)\|_
				{\mathbb{H}_{\mathrm{obs}}}^{2}
				\right)+\|\gamma_{-}(\boldsymbol{u}_1-\boldsymbol{u}_2)\|_
				{[H^{1}(\Gamma_{\mathrm{in}})]^{3}}^{2} \Bigr], \end{aligned}
			\label{estimat-stab}
		\end{equation}
		where $C>0$ is a constant independent of $\boldsymbol{u}_1$ and  $%
		\boldsymbol{u}_2$.
	\end{theorem}
	
	\begin{remark}[Hybrid Error Propagation]
		The stability estimate~\eqref{estimat-stab} reflects the hybrid 
		parabolic--hyperbolic nature of the inverse problem~\eqref{problem-inverse}.
		The error in the inflow boundary data propagates into the domain with 
		Lipschitz stability, as expected for the underlying first-order transport 
		equation along the characteristics of $\boldsymbol{F}$. In contrast, the 
		error in the parabolic observation $\mathcal{L}(\boldsymbol{y}_1-  
		\boldsymbol{y}_2)$ is transmitted through the modulus $\Theta$, which may 
		be Lipschitz, H\"{o}lder or logarithmic  \cite%
		{kian2020logarithmic,huang2020stability,imanuvilov2026stability}  for $\mathcal{L}$ given in~\eqref{internal-F}--\eqref{Final-F},  reflecting the
		ill-posedness of the associated parabolic inverse source  problem (these references handle the Dirichlet boundary condition case). Thus, the
		overall stability combines a Lipschitz contribution from  transport with a
		potentially weaker contribution from parabolic source  recovery. Since $%
		\boldsymbol{b}_1-\boldsymbol{b}_2=  \boldsymbol{y}_1-\boldsymbol{y}_2$, the
		same modulus applies to differences  of magnetic observations.
	\end{remark}
		\subsection*{Stability from the terminal observation}
	
	
	The injectivity result of Theorem~\ref{thm:injectivity-beta}(ii) shows that,
	under the nonvanishing condition \eqref{non-vanishing}, the terminal
	observation $\boldsymbol{y}(T)$ determines $\boldsymbol{z}$ uniquely. We now
	quantify this determination. Throughout this section we work with the
	terminal observation $\mathcal{L}(\boldsymbol{y})=\boldsymbol{y}(T)\in D(\Delta_{\mathrm{H}})$. The case of subdomain observation $\omega_{\mathrm{obs}}$ needs more car.

	\begin{theorem}
		\label{thm:lip-stab} Assume that there exists a constant $c_{0}>0$ such that 
		\begin{equation}
			|\gamma _{n}|\geq \frac{c_{0}}{1+\lambda _{n}},\qquad n\geq 1,
			\label{lower-bound}
		\end{equation}
		where $\gamma _{n}$ is defined in \ref{non-vanishing}. Then for every $%
		\boldsymbol{z}\in \mathbb{H}$, the associated solution $\boldsymbol{y}$ of %
		\eqref{eq:abstract-inverseh} satisfies 
		\begin{equation}  \label{eq:lip-stab}
			\|\boldsymbol{z}\|_{\mathbb{H}}^{2}\leq C\Big(\|\boldsymbol{y}(T)\|_{\mathbb{%
					H}}^{2} +\|\Delta_{\mathrm{H}}\boldsymbol{y}(T)\|_{\mathbb{H}}^{2}\Big),
		\end{equation}
		for some constant $C>0$ independent of $\boldsymbol{z}$.
	\end{theorem}
	
	\begin{proof}
		We write \(\boldsymbol{y}(T)\) in the spectral representation
		\[
		\boldsymbol{y}(T)
		=
		\sum_{n=1}^{\infty}\sum_{k=1}^{m_n}
		\gamma_n\,z_{n,k}\,\boldsymbol{\phi}_{n,k}.
		\]
		Using \eqref{lower-bound} coefficientwise, for each \((n,k)\),
		\[
		|z_{n,k}|
		=
		\frac{|z_{n,k}\gamma_n|}{|\gamma_n|}
		\le
		\frac{1+\lambda_n}{c_0}\,|z_{n,k}\gamma_n|.
		\]
		Squaring and summing over \(n,k\) yields
		\begin{align*}
			\|\boldsymbol{z}\|_{\mathbb{H}}^{2}
			&=
			\sum_{n=1}^{\infty}\sum_{k=1}^{m_n}|z_{n,k}|^2
			\le
			\frac{1}{c_0^2}
			\sum_{n=1}^{\infty}\sum_{k=1}^{m_n}
			(1+\lambda_n)^2\,|z_{n,k}\gamma_n|^2 \\[1mm]
			&=
			\frac{1}{c_0^2}\,
			\|(I-\Delta_{\mathrm{H}})\boldsymbol{y}(T)\|_{\mathbb{H}}^2
			\le
			C\Big(
			\|\boldsymbol{y}(T)\|_{\mathbb{H}}^2
			+
			\|\Delta_{\mathrm{H}}\boldsymbol{y}(T)\|_{\mathbb{H}}^2
			\Big),
		\end{align*}
		where we used that
		\(-\Delta_{\mathrm{H}}\boldsymbol{\phi}_{n,k}
		=\lambda_n\boldsymbol{\phi}_{n,k}\), and the last inequality follows from
		\((1+\lambda)^2\le 2(1+\lambda^2)\) for all \(\lambda\ge 0\).
	\end{proof}
	
	\begin{remark}[Example of a time profile satisfying \eqref{lower-bound}]
		Take the constant profile 
		\begin{equation*}
			\alpha (t)\equiv 1,\qquad t\in \lbrack 0,T].
		\end{equation*}%
		Then, for $n\geq 1$, 
		\begin{equation*}
			\gamma _{n}=\int_{0}^{T}e^{-\lambda _{n}(T-s)}\,ds=\int_{0}^{T}e^{-\lambda
				_{n}u}\,du=\frac{1-e^{-\lambda _{n}T}}{\lambda _{n}}\geq \frac{1-e^{-\lambda
					_{1}T}}{\lambda _{n}+1},
		\end{equation*}%
		Since $\lambda _{n}>0$, we have $\gamma _{n}>0$ for every $n\geq 1$, so the
		non-vanishing condition \eqref{non-vanishing} holds.
	\end{remark}

	\subsection{Uniqueness Result}
	
	We now state the main uniqueness result for the inverse problem~%
	\eqref{problem-inverse}. The result follows from the injectivity of the
	observation of the auxiliary variable and the \emph{$\Omega$-filling}
	property of the flow generated by $\boldsymbol{F}$.
	
	\begin{theorem}
		\label{main-theorem}  Assume that $\boldsymbol{F}$ satisfies~%
		\eqref{F-assumption}. Suppose, in  addition, that:
		
		\begin{enumerate}
			\item[$\mathrm{(i)}$] $\alpha\in L^{2}(0,T)$ is such that the operator  
			\begin{eqnarray}
				\mathcal{O}&:&\mathbb{H}\longrightarrow\mathbb{H}_{\mathrm{obs}},  \label{O}
				\\
				\mathcal{O}(\boldsymbol{z})&=&\mathcal{L}(\boldsymbol{y}(\boldsymbol{z})), 
				\notag
			\end{eqnarray}
			is injective, where $\boldsymbol{y}(\boldsymbol{z})$ denotes the solution 
			of~\eqref{eq:abstract-inverseh} associated with $\boldsymbol{z}$ and 
			initial datum $\boldsymbol{0}$, and $\mathcal{L}$ is the observation 
			operator defined in~\eqref{F};
			
			\item[$\mathrm{(ii)}$] The background vector field $\boldsymbol{F}$ 
			satisfies~\eqref{F-assumption} and the flow $\boldsymbol{X}$ is  \emph{$%
				\Omega$-filling and one-to-one} in the sense of  Definition~\ref%
			{Definition-Omega-Filing}. 
		\end{enumerate}
		
		Then the observation operator $\Lambda_{\mathcal{L}}$ defined  in~%
		\eqref{L_gen} is injective. More precisely, if  
		\begin{equation}
			\mathcal{L}(\boldsymbol{y}_1)=\mathcal{L}(\boldsymbol{y}_2) \qquad\text{and}%
			\qquad \gamma_{-}\boldsymbol{u}_1=\gamma_{-}\boldsymbol{u}_2 \quad\text{on }%
			\Gamma_{\mathrm{in}},  \label{eq:uniqueness-data}
		\end{equation}
		then  
		\begin{equation}
			\boldsymbol{u}_1=\boldsymbol{u}_2.  \label{eq:uniqueness-conclusion}
		\end{equation}
	\end{theorem}
	
	By Theorem~\ref{main-theorem}, the source $\boldsymbol{u}$ is uniquely
	determined by the state observations $\mathcal{L}(\boldsymbol{y})$ together
	with the Dirichlet trace $\gamma_{-}\boldsymbol{u}$ on the inflow boundary.
	Also, for the uniqueness result, the compatibility condition~$\mathrm{(ii)}$
	of Theorem~\ref{main-stability} is not required.
	
	\begin{remark}[Gauge fixing and hybrid parabolic--hyperbolic uniqueness]
		\label{rem:parabolic-hyperbolic-uniqueness}  The measurement $\Lambda_{%
			\mathcal{L}}$ defined in~\eqref{L_gen} combines  two mechanisms. The
		parabolic observation $\mathcal{L}$ uniquely determines  $\boldsymbol{z}$,
		where  $\alpha\boldsymbol{z}=\nabla\times(\boldsymbol{u}\times\boldsymbol{F}%
		)-\Delta\boldsymbol{F}$.  However, the operator  $\boldsymbol{u}%
		\mapsto\nabla\times(\boldsymbol{u}\times\boldsymbol{F})$  has a nontrivial
		kernel: any component of $\boldsymbol{v}$ collinear with  $\boldsymbol{F}$
		is annihilated and remains invisible to the state  observations, as shown in~%
		\eqref{parallel}. The hyperbolic inflow trace  $\gamma_{-}(\boldsymbol{u})$
		propagates along the  characteristics of $\boldsymbol{F}$ and fixes
		precisely this missing  component, thereby removing the gauge freedom. Hence
		neither measurement  alone suffices, while together they yield uniqueness in
		Theorem~\ref{main-theorem}.
	\end{remark}
	
	\begin{remark}[Applications to partial observations]
		The abstract observation operator $\mathcal{L}$ allows  Theorem~\ref%
		{main-theorem} to be applied to several measurement  configurations for the
		auxiliary parabolic  problems~\eqref{eq:abstract-inverseh}. For example, one
		may consider the observations~\eqref{internal-F}--\eqref{Final-F}. Such
		observation  configurations are studied in inverse source problems for
		parabolic  equations; see, e.g.,  \cite%
		{yamamoto2009carleman,lin2022theoretical,imanuvilov2026stability}.  Whenever
		the corresponding observation determines the spatial source  difference $%
		\boldsymbol{z}_1-\boldsymbol{z}_2$ uniquely (or satisfies the  stability
		estimate assumed in Theorem~\ref{main-stability}), the theorem  applies
		directly.
	\end{remark}
	
	\subsection*{Application of the uniqueness result}
	
	For a bounded connected domain $\Omega$ with sufficiently regular boundary, $%
	-\Delta_{\mathrm{H}}$ is a nonnegative self-adjoint operator on $\mathbb{H}$
	with compact resolvent \cite{ferraresso2026note}. Consequently, its spectrum
	is discrete. Denote its distinct eigenvalues by 
	\begin{equation*}
		0<\lambda_1\leq\lambda_2\leq\cdots, \qquad \lambda_n\longrightarrow+\infty,
	\end{equation*}
	and let 
	\begin{equation*}
		E_{\lambda_n}=\ker(	-\Delta_{\mathrm{H}}-\lambda_n I)
	\end{equation*}
	be the corresponding finite-dimensional eigenspaces. For each $n\geq1$, let 
	\begin{equation*}
		\{\boldsymbol{\phi}_{n,k}:1\leq k\leq m_n\}
	\end{equation*}
	be an orthonormal basis of $E_{\lambda_n}$, where $m_n=\dim E_{\lambda_n}$.
	The eigenfunctions satisfy the absolute boundary conditions 
	\begin{equation*}
		\boldsymbol{\phi}_{n,k}\cdot\boldsymbol{n}=0, \qquad (\nabla\times 
		\boldsymbol{\phi}_{n,k})\times\boldsymbol{n}=0 \quad\text{on }
		\partial\Omega, \qquad \nabla\cdot\boldsymbol{\phi}_{n,k}=0.
	\end{equation*}
	
	For each $\boldsymbol{z}\in \mathbb{H}$. Let $\boldsymbol{y}(\boldsymbol{z})$
	be the solution to system \eqref{eq:abstract-inverseh} and let $\mathcal{O}$
	be the operator defined in \ref{O}. We consider the following two
	observation configurations defined in \ref{internal-F} and \ref{Final-F}.
	
	\begin{theorem}
		\label{thm:injectivity-beta}
		Let \(\Omega\subset\mathbb{R}^3\) be bounded and connected, and let
		\(\alpha\in L^\infty(0,T)\).
		
		\begin{enumerate}
			\item[(i)] Assume that \(\omega_{\mathrm{obs}}\) is a non-empty open set,
			that \(\alpha\not\equiv 0\), and that
			\begin{equation*}
				\operatorname{supp}\alpha \subset [0,T_0]
				\qquad \text{for some }T_0<T.
			\end{equation*}
			Then \(\mathcal{O}\) defined in \ref{O} is injective for the observation
			\(\mathcal{L}\) defined in \ref{internal-F}.
			
			\item[(ii)] Suppose that
			\begin{equation}
				\gamma_i:=\int_0^T e^{-\lambda_i(T-s)}\alpha(s)\,ds\neq0,
				\qquad i\geq 1.
				\label{non-vanishing}
			\end{equation}
			Then \(\mathcal{O}\) defined in \ref{O} is injective for the terminal
			observation \(\mathcal{L}\) defined in \ref{Final-F}:
			\begin{equation*}
				\mathcal{O}(\boldsymbol{y})=\boldsymbol{y}(T)
				\quad \text{in }\mathbb{H}.
			\end{equation*}
		\end{enumerate}
	\end{theorem}
	
	\begin{proof}
		Suppose that
		\[
		\mathcal{O}(\boldsymbol{z}_1)=\mathcal{O}(\boldsymbol{z}_2).
		\]
		Set
		\[
		\boldsymbol{z}=\boldsymbol{z}_1-\boldsymbol{z}_2,
		\qquad
		\boldsymbol{w}
		=\boldsymbol{y}(\boldsymbol{z}_1)-\boldsymbol{y}(\boldsymbol{z}_2).
		\]
		Since the two solutions have the same initial condition, their difference
		satisfies
		\begin{equation}
			\label{eq:difference}
			\boldsymbol{w}'(t)-\Delta_{\mathrm H}\boldsymbol{w}(t)
			=\alpha(t)\boldsymbol{z},
			\qquad
			\boldsymbol{w}(0)=\boldsymbol{0}.
		\end{equation}
		
		Recall that \(-\Delta_{\mathrm H}\) is a nonnegative self-adjoint operator on
		\(\mathbb{H}\) with compact resolvent. Let
		\(\{\lambda_i\}_{i\geq 1}\) denote the distinct eigenvalues of
		\(-\Delta_{\mathrm H}\). For each \(i\geq 1\), let \(E_{\lambda_i}\) be the
		corresponding eigenspace, of dimension \(m_i\), and let
		\(\{\boldsymbol{\phi}_{i,k}\}_{k=1}^{m_i}\) be an orthonormal basis of
		\(E_{\lambda_i}\). Write
		\[
		\boldsymbol{z}
		=\sum_{i\geq 1}\sum_{k=1}^{m_i}
		z_{i,k}\boldsymbol{\phi}_{i,k},
		\qquad
		z_{i,k}
		=(\boldsymbol{z},\boldsymbol{\phi}_{i,k})_{\mathbb{H}}.
		\]
		By Duhamel's formula,
		\begin{equation}
			\label{eq:w-eigen-expansion}
			\boldsymbol{w}(t)
			=\sum_{i\geq 1}\sum_{k=1}^{m_i}
			z_{i,k}
			\left(\int_0^t e^{-\lambda_i(t-s)}
			\alpha(s)\,ds\right)\boldsymbol{\phi}_{i,k}.
		\end{equation}
		The series converges in \(\mathbb{H}\) for every \(t\in[0,T]\).
		
		\medskip
		\noindent
		\textbf{Case (i): observation over a time interval.}
		
		Assume that
		\[
		\mathcal{L}\boldsymbol{w}(t)=\boldsymbol{0}
		\quad\text{for almost every }t\in(0,T).
		\]
		We prove that \(\boldsymbol{z}=\boldsymbol{0}\).
		
		Since \(\alpha\) is supported in \([0,T_0]\), with \(T_0<T\), equation
		\eqref{eq:difference} is homogeneous for \(t>T_0\). Thus
		\[
		\boldsymbol{w}(t)
		=e^{(t-T_0)\Delta_{\mathrm H}}\boldsymbol{w}(T_0),
		\qquad t>T_0.
		\]
		Since \(-\Delta_{\mathrm H}\) is nonnegative and self-adjoint,
		\(\Delta_{\mathrm H}\) generates the analytic contraction semigroup
		\(e^{t\Delta_{\mathrm H}}\). Consequently,
		\[
		t\longmapsto\mathcal{L}\boldsymbol{w}(t)
		\]
		is real analytic on \((T_0,\infty)\).
		
		The observation vanishes almost everywhere on \((T_0,T)\). By continuity, it
		vanishes everywhere on this interval. The identity theorem for real-analytic
		functions therefore gives
		\[
		\mathcal{L}\boldsymbol{w}(t)=\boldsymbol{0},
		\qquad t>T_0.
		\]
		Together with the original observation equality, this yields
		\[
		\mathcal{L}\boldsymbol{w}(t)=\boldsymbol{0}
		\quad\text{for almost every }t>0.
		\]
		
		The solution \(\boldsymbol{w}\) is bounded on \([0,\infty)\). Indeed, it is
		continuous on \([0,T_0]\), and for \(t>T_0\),
		\[
		\|\boldsymbol{w}(t)\|_{\mathbb{H}}
		\leq \|\boldsymbol{w}(T_0)\|_{\mathbb{H}},
		\]
		by contractivity of the semigroup. Hence, for \(\operatorname{Re}p>0\), the
		Laplace transform
		\[
		\widehat{\boldsymbol{w}}(p)
		=\int_0^\infty e^{-pt}\boldsymbol{w}(t)\,dt
		\]
		exists in \(\mathbb{H}\). Also, since \(\alpha\) is supported in \([0,T_0]\),
		\[
		\widehat{\alpha}(p)
		=\int_0^{T_0}e^{-pt}\alpha(t)\,dt.
		\]
		Taking the Laplace transform of \eqref{eq:difference} gives
		\[
		(pI-\Delta_{\mathrm H})\widehat{\boldsymbol{w}}(p)
		=\widehat{\alpha}(p)\boldsymbol{z}.
		\]
		Therefore,
		\[
		\widehat{\boldsymbol{w}}(p)
		=\widehat{\alpha}(p)(pI-\Delta_{\mathrm H})^{-1}\boldsymbol{z},
		\qquad \operatorname{Re}p>0.
		\]
		
		Since \(\mathcal{L}\) is bounded and
		\(\mathcal{L}\boldsymbol{w}(t)=\boldsymbol{0}\) almost everywhere, we obtain
		\[
		\widehat{\alpha}(p)\,
		\mathcal{L}(pI-\Delta_{\mathrm H})^{-1}\boldsymbol{z}
		=\boldsymbol{0},
		\qquad \operatorname{Re}p>0.
		\]
		
		The function \(\widehat{\alpha}\) is entire and is not identically zero because
		\(\alpha\not\equiv0\). Its nonzero set therefore contains a nonempty open
		subset of the right half-plane. Consequently,
		\[
		\mathcal{L}(pI-\Delta_{\mathrm H})^{-1}\boldsymbol{z}
		=\boldsymbol{0}
		\]
		on that open subset.
		
		The function
		\[
		p\longmapsto
		\mathcal{L}(pI-\Delta_{\mathrm H})^{-1}\boldsymbol{z}
		\]
		is holomorphic on
		\[
		\mathbb{C}\setminus\{-\lambda_i:i\geq 1\}.
		\]
		This set is connected, so the identity theorem yields
		\[
		\mathcal{L}(pI-\Delta_{\mathrm H})^{-1}\boldsymbol{z}
		=\boldsymbol{0},
		\qquad
		p\in\mathbb{C}\setminus\{-\lambda_i:i\geq 1\}.
		\]
		
		Let \(P_i\) be the orthogonal projection onto \(E_{\lambda_i}\). Near
		\(p=-\lambda_i\), the resolvent \((pI-\Delta_{\mathrm H})^{-1}\) has the
		expansion
		\[
		(pI-\Delta_{\mathrm H})^{-1}
		=\frac{P_i}{p+\lambda_i}+\mathcal{R}_i(p),
		\]
		where \(\mathcal{R}_i\) is holomorphic near \(p=-\lambda_i\).
		Taking the residue at \(p=-\lambda_i\) gives
		\[
		\mathcal{L}P_i\boldsymbol{z}=\boldsymbol{0}.
		\]
		Equivalently,
		\[
		\mathcal{L}\left(
		\sum_{k=1}^{m_i}z_{i,k}\boldsymbol{\phi}_{i,k}
		\right)=\boldsymbol{0}.
		\]
		
		Since \(\omega_{\mathrm{obs}}\subset\Omega\) is a non-empty open set, the
		elliptic unique continuation property \cite{saut1982unicite} gives
		\[
		\sum_{k=1}^{m_i}z_{i,k}\boldsymbol{\phi}_{i,k}
		=\boldsymbol{0}
		\quad\text{in }\mathbb{H}.
		\]
		By orthonormality, \(z_{i,k}=0\) for every \(k\). Since \(i\) is arbitrary,
		all coefficients vanish and
		\[
		\boldsymbol{z}=\boldsymbol{0}.
		\]
		Therefore,
		\[
		\boldsymbol{z}_1=\boldsymbol{z}_2.
		\]
		
		\medskip
		\noindent
		\textbf{Case (ii): terminal observation.}
		
		Suppose instead that equality of the terminal observations gives
		\[
		\boldsymbol{w}(T)=\boldsymbol{0}
		\quad\text{in }\mathbb{H}.
		\]
		Evaluating \eqref{eq:w-eigen-expansion} at \(t=T\), we obtain
		\[
		\boldsymbol{0}
		=\sum_{i\geq 1}\sum_{k=1}^{m_i}
		\gamma_i z_{i,k}\boldsymbol{\phi}_{i,k},
		\qquad
		\gamma_i
		=\int_0^T e^{-\lambda_i(T-s)}\alpha(s)\,ds.
		\]
		Taking the inner product with each basis element yields
		\[
		\gamma_i z_{i,k}=0,
		\qquad i\geq 1,\quad k=1,\ldots,m_i.
		\]
		Assume the nonvanishing moment condition
		\[
		\gamma_i\neq0
		\qquad\text{for every }i\geq 1.
		\]
		Then \(z_{i,k}=0\) for every \(i,k\), so \(\boldsymbol{z}=\boldsymbol{0}\).
		Consequently,
		\[
		\boldsymbol{z}_1=\boldsymbol{z}_2.
		\]
		
		Thus the observation operator is injective in both cases.
	\end{proof}
	

	
	
	\section{Proof of main results}
	\label{sec:proof-main-results}
	
	We prove the stability and uniqueness results for the time-dependent
	velocity field. Since the background field $\boldsymbol F$ is independent
	of time, the stationary curl--div estimate of
	Proposition~\ref{Proposition2} can be applied for almost every
	$t\in(0,T)$, provided its compatibility and regularity assumptions hold.
	
	By linearity, it suffices to consider the difference of two solutions
	with the same initial magnetic data and the same background field
	$\boldsymbol F$. The known term $\Delta\boldsymbol F$ and the common
	initial contribution then cancel.
	
	For $i=1,2$, let
	\[
	\boldsymbol v_i(t,\boldsymbol x)
	=\alpha(t)\boldsymbol u_i(\boldsymbol x),
	\qquad
	\boldsymbol z_i
	=\nabla\times(\boldsymbol u_i\times\boldsymbol F).
	\]
	Set
	\[
	\boldsymbol u=\boldsymbol u_1-\boldsymbol u_2,
	\qquad
	\boldsymbol v=\boldsymbol v_1-\boldsymbol v_2
	=\alpha\boldsymbol u,
	\qquad
	\boldsymbol z=\boldsymbol z_1-\boldsymbol z_2
	=\nabla\times(\boldsymbol u\times\boldsymbol F).
	\]
	Consequently, the difference source is
	\[
	\boldsymbol h(t)=\alpha(t)\boldsymbol z,
	\qquad
	\boldsymbol g(t)=\gamma_-\boldsymbol v(t)
	=\alpha(t)\gamma_-\boldsymbol u.
	\]
	Throughout the proof, we assume
	$\alpha\in L^2(0,T)$ and
	$\|\alpha\|_{L^2(0,T)}>0$.
	
	\subsection{Proof of the stability result}
	
	\begin{proof}[Proof of Theorem~\protect\ref{main-stability}]
		By the regularity assumption in Theorem~\ref{main-stability},
		\[
		\gamma_-\boldsymbol u
		\in[H^1(\Gamma_{\mathrm{in}})]^3.
		\]
		Hence, for almost every $t\in(0,T)$,
		\[
		\boldsymbol g(t)
		=\alpha(t)\gamma_-\boldsymbol u
		\in[H^1(\Gamma_{\mathrm{in}})]^3.
		\]
		Moreover, the compatibility condition is preserved under multiplication
		by $\alpha(t)$. Indeed, by linearity of $\mathcal C$ and the assumption
		$\mathcal C(\gamma_-\boldsymbol u)=0$,
		\[
		\mathcal C(\boldsymbol g(t))
		=\alpha(t)\mathcal C(\gamma_-\boldsymbol u)=0
		\quad\text{for a.e. }t\in(0,T).
		\]
		
		Therefore, Proposition~\ref{Proposition2} applies to
		$\boldsymbol v(t)$ for almost every $t\in(0,T)$ and gives
		\begin{equation}
			\label{eq:pointwise-stability}
			\|\boldsymbol v(t)\|_{\mathbb X_{\boldsymbol F}}^2
			\leq C\left(
			\|\boldsymbol h(t)\|_{\mathbb H}^2
			+
			\|\boldsymbol g(t)\|_{[H^1(\Gamma_{\mathrm{in}})]^3}^2
			\right),
		\end{equation}
		where $C>0$ is independent of $t$.
		
		Integrating \eqref{eq:pointwise-stability} over $(0,T)$ and using
		$\boldsymbol v=\alpha\boldsymbol u$,
		$\boldsymbol h=\alpha\boldsymbol z$, and
		$\boldsymbol g=\alpha\gamma_-\boldsymbol u$, we obtain
		\[
		\begin{aligned}
			\|\alpha\|_{L^2(0,T)}^2
			\|\boldsymbol u\|_{\mathbb X_{\boldsymbol F}}^2
			\leq C\|\alpha\|_{L^2(0,T)}^2
			\bigl(
			&\|\boldsymbol z\|_{\mathbb H}^2+\|\gamma_-\boldsymbol u\|_
			{[H^1(\Gamma_{\mathrm{in}})]^3}^2
			\bigr).
		\end{aligned}
		\]
		Since $\|\alpha\|_{L^2(0,T)}>0$, division by
		$\|\alpha\|_{L^2(0,T)}^2$ yields
		\begin{equation}
			\label{eq:u-before-parabolic}
			\|\boldsymbol u\|_{\mathbb X_{\boldsymbol F}}^2
			\leq C\left(
			\|\boldsymbol z\|_{\mathbb H}^2
			+
			\|\gamma_-\boldsymbol u\|_
			{[H^1(\Gamma_{\mathrm{in}})]^3}^2
			\right).
		\end{equation}
		
		By assumption (i) of Theorem~\ref{main-stability}, the auxiliary
		inverse-source problem for the difference field satisfies
		\begin{equation}
			\label{eq:h-stability-time-dependent}
			\|\boldsymbol z\|_{\mathbb H}^2
			\leq
			\Theta\left(
			\|\mathcal L(\boldsymbol y_1-\boldsymbol y_2)\|_
			{\mathbb H_{\mathrm{obs}}}^2
			\right),
		\end{equation}
		where $\boldsymbol y_i$ is the solution of
		\eqref{eq:abstract-inverseh} associated with $\boldsymbol z_i$.
		
		Since the two magnetic perturbations have the same initial data and
		background field, their difference satisfies
		\[
		\boldsymbol b_1-\boldsymbol b_2
		=\boldsymbol y_1-\boldsymbol y_2.
		\]
		By linearity of $\mathcal L$,
		\begin{equation}
			\label{eq:obs-difference}
			\mathcal L(\boldsymbol y_1-\boldsymbol y_2)
			=\mathcal L(\boldsymbol b_1)-\mathcal L(\boldsymbol b_2).
		\end{equation}
		Combining \eqref{eq:u-before-parabolic},
		\eqref{eq:h-stability-time-dependent}, and
		\eqref{eq:obs-difference}, we conclude that
		\[
		\begin{aligned}
			\|\boldsymbol u_1-\boldsymbol u_2\|_
			{\mathbb X_{\boldsymbol F}}^2
			\leq C\Bigl[
			&\Theta\Bigl(
			\|\mathcal L(\boldsymbol b_1)
			-\mathcal L(\boldsymbol b_2)\|_
			{\mathbb H_{\mathrm{obs}}}^2
			\Bigr)+	\|\gamma_-(\boldsymbol u_1-\boldsymbol u_2)\|_
			{[H^1(\Gamma_{\mathrm{in}})]^3}^2
			\Bigr].
		\end{aligned}
		\]
		This is the asserted estimate \eqref{estimat-stab}.
	\end{proof}

	\subsection{Proof of the uniqueness result}
	
	\begin{proof}[Proof of Theorem~\protect\ref{main-theorem}]
		Suppose that
		\[
		\Lambda_{\mathcal L}(\boldsymbol v_1)
		=\Lambda_{\mathcal L}(\boldsymbol v_2).
		\]
		By the definition of $\Lambda_{\mathcal L}$, we have
		\[
		\mathcal L(\boldsymbol b_1)=\mathcal L(\boldsymbol b_2),
		\qquad
		\gamma_-\boldsymbol v_1=\gamma_-\boldsymbol v_2.
		\]
		Set
		\[
		\boldsymbol u=\boldsymbol u_1-\boldsymbol u_2,
		\qquad
		\boldsymbol z
		=\nabla\times(\boldsymbol u\times\boldsymbol F)
		=\boldsymbol z_1-\boldsymbol z_2.
		\]
		
		As in the stability proof, the common initial contribution and known
		terms cancel, and therefore
		\[
		\boldsymbol b_1-\boldsymbol b_2
		=\boldsymbol y_1-\boldsymbol y_2.
		\]
		It follows that
		\[
		\mathcal O(\boldsymbol z)
		=\mathcal L(\boldsymbol y_1-\boldsymbol y_2)
		=\mathcal L(\boldsymbol b_1)-\mathcal L(\boldsymbol b_2)
		=0.
		\]
		By assumption (i) of Theorem~\ref{main-theorem}, the operator
		$\mathcal O$ is injective. Consequently,
		\[
		\boldsymbol z=0,
		\qquad\text{that is,}\qquad
		\nabla\times(\boldsymbol u\times\boldsymbol F)=0
		\quad\text{in }\Omega.
		\]
		
		Furthermore, since
		$\boldsymbol v_i=\alpha\boldsymbol u_i$, the equality of the inflow
		traces gives
		\[
		\alpha(t)\gamma_-\boldsymbol u=0
		\quad\text{in }
		L^2\bigl(0,T;
		L^2(\Gamma_{\mathrm{in}};
		\tau|\boldsymbol F\cdot\boldsymbol n|\,d\sigma)\bigr).
		\]
		Thus
		\[
		\begin{aligned}
			0
			&=\|\alpha\gamma_-\boldsymbol u\|_
			{L^2(0,T;L^2(\Gamma_{\mathrm{in}};
				\tau|\boldsymbol F\cdot\boldsymbol n|\,d\sigma))}^2=\|\alpha\|_{L^2(0,T)}^2
			\|\gamma_-\boldsymbol u\|_
			{L^2(\Gamma_{\mathrm{in}};
				\tau|\boldsymbol F\cdot\boldsymbol n|\,d\sigma)}^2.
		\end{aligned}
		\]
		Since $\|\alpha\|_{L^2(0,T)}>0$, we conclude that
		\[
		\gamma_-\boldsymbol u=0
		\quad\text{in }
		L^2(\Gamma_{\mathrm{in}};
		\tau|\boldsymbol F\cdot\boldsymbol n|\,d\sigma).
		\]
		
		We have therefore obtained
		\[
		\nabla\times(\boldsymbol u\times\boldsymbol F)=0,
		\qquad
		\gamma_-\boldsymbol u=0.
		\]
		Applying Proposition~\ref{Proposition2} with zero curl data and zero
		inflow data gives
		\[
		\|\boldsymbol u\|_{\mathbb X_{\boldsymbol F}}=0.
		\]
		Hence $\boldsymbol u_1=\boldsymbol u_2$, and consequently $\Lambda_{\mathcal L}$ is injective, proving the uniqueness result.
	\end{proof}
	
	\begin{remark}[Weaker trace observation]
		\label{rem:weaker-trace}
		For the uniqueness result (Theorem \ref{main-theorem}), the observation operator can be weakened to
		\begin{equation}
			\begin{aligned}
				\Lambda_{\mathcal{L}}:\mathbb{X}_{\boldsymbol{F}}
				&\longrightarrow
				\mathbb{H}_{\mathrm{obs}}\times
				L^{2}\bigl(\Gamma_{\mathrm{in}};\tau|\boldsymbol{F}\cdot\boldsymbol{n}|\,d\sigma\bigr),\\
				\boldsymbol{u}
				&\longmapsto
				\bigl(\mathcal{L}(\boldsymbol{b}),\,
				\gamma_{-}\boldsymbol{u}\bigr),
			\end{aligned}
			\label{L_gen_weak}
		\end{equation}
		since \(\gamma_{-}\) is bounded into this weighted space by
		Proposition~\ref{proposition-trace}. The \(H^{1}(\Gamma_{\mathrm{in}})\)
		regularity was used only to define \(\mathcal{C}(\boldsymbol{g})\) in
		\(L^{2}(\Gamma_{\mathrm{in}})\), and it is not needed in the uniqueness
		proof: the trace data there are zero, and \(\mathcal{C}(\boldsymbol{0})=0\)
		trivially. Hence uniqueness holds for \eqref{L_gen_weak}. 
	\end{remark}
	
	
	\section{Numerical Simulations}
	\label{sec:numerical}
	
	We present three-dimensional numerical experiments comparing two
	reconstruction strategies for the spatial velocity field
	$\boldsymbol u$, under the separable ansatz
	\[
	\boldsymbol v(t,\boldsymbol x)=\alpha(t)\boldsymbol u(\boldsymbol x),
	\]
	where $\alpha\in L^2(0,T)$ is known. Throughout the numerical
	experiments, we take $\alpha(t)\equiv1$, so that
	$\boldsymbol v=\boldsymbol u$. The source term is therefore
	\[
	\boldsymbol h
	=\Delta\boldsymbol F+\boldsymbol z,
	\qquad
	\boldsymbol z=\nabla\times(\boldsymbol u\times\boldsymbol F).
	\]
	
	The first strategy is the two-stage approach developed above. It first
	reconstructs the auxiliary source $\boldsymbol z$ from a terminal
	observation of the auxiliary state $\boldsymbol y$, and then recovers
	$\boldsymbol u$ by solving the curl--div problem~\eqref{eq:div-curl}.
	The second strategy reconstructs $\boldsymbol u$ directly from the
	terminal magnetic observation by solving a regularized inverse problem
	subject to the divergence and inflow constraints
	\[
	\nabla\cdot\boldsymbol u=0\quad\text{in }\Omega,
	\qquad
	\gamma_-\boldsymbol u=\boldsymbol g_{\mathrm{obs}}
	\quad\text{on }\Gamma_{\mathrm{in}}.
	\]
	
	We consider the terminal observation operator
	\[
	\mathcal L(\boldsymbol b)=\boldsymbol b(T)\in\mathbb H.
	\]
	Here $\boldsymbol b(\boldsymbol u)$ denotes the solution of
	\eqref{b-system} generated by
	$\boldsymbol v=\alpha\boldsymbol u$. The objective is to reconstruct
	$\boldsymbol u$ from the terminal magnetic data and the prescribed
	inflow trace.
	
	\subsection{Two-stage reconstruction}
	\label{subsec:two-stage}
	
	In the first stage, we reconstruct the auxiliary source
	$\boldsymbol z$ from the terminal observation of the auxiliary state
	$\boldsymbol y$. Since $\alpha$ and $\Delta\boldsymbol F$ are known,
	the source in the magnetic equation is determined by
	$\Delta\boldsymbol F+\alpha\boldsymbol z$. For the auxiliary problem
	\eqref{eq:abstract-inverseh}, however, the unknown source is
	$\alpha\boldsymbol z$.
	
	Accordingly, $\boldsymbol y_{\mathrm{obs}}(T)$ denotes the terminal
	observation after accounting for the known contribution associated with
	$\Delta\boldsymbol F$ and the prescribed initial data. This distinction
	is necessary when the auxiliary state $\boldsymbol y$ is defined by
	subtracting those known contributions from the magnetic state
	$\boldsymbol b$.
	
	The reconstructed source must be compatible with the prescribed inflow
	trace in order for the second-stage curl--div problem to be solvable.
	
	For a tolerance parameter $\varepsilon\geq0$, define
	\[
	\mathbb G_{\mathrm{adm}}
	=
	\left\{
	\boldsymbol g\in[H^1(\Gamma_{\mathrm{in}})]^3:
	\|\mathcal C(\boldsymbol g)\|_{L^2(\Gamma_{\mathrm{in}})}
	\leq\varepsilon
	\right\},
	\]
	where $\mathcal C$ is the compatibility operator defined in
	\eqref{C(g,h)}, with its dependence on the data understood as specified
	in that definition.
	
	The exact compatibility condition is
	$\mathcal C(\boldsymbol g)=0$. Thus, if $\varepsilon>0$, membership
	in $\mathbb G_{\mathrm{adm}}$ expresses approximate compatibility only;
	it does not, by itself, guarantee exact solvability of the curl--div
	problem. If the measured boundary data fail the exact condition, one
	must either use compatible boundary data or specify a consistent
	approximation of the compatibility constraint.
	
	For the terminal observation, we reconstruct $\boldsymbol z$ by solving
	the Tikhonov-regularized minimization problem
	\begin{equation}
		\label{eq:two-stage-source-minimization}
		\boldsymbol z^\ast
		=
		\underset{\boldsymbol z\in\mathbb H}{\operatorname{argmin}}
		\left\{
		\frac12
		\|\boldsymbol y(\boldsymbol z)(T)
		-\boldsymbol y_{\mathrm{obs}}(T)\|_{\mathbb H}^2
		+
		\frac{\gamma}{2}\|\boldsymbol z\|_{\mathbb H}^2
		\right\},
	\end{equation}
	where $\gamma>0$ is the regularization parameter and
	$\boldsymbol y(\boldsymbol z)$ solves
	\eqref{eq:abstract-inverseh} with source $\alpha\boldsymbol z$.
	
	In the second stage, we solve the curl--div problem
	\eqref{eq:div-curl} with source $\boldsymbol z^\ast$ and the prescribed
	inflow datum $\boldsymbol g_{\mathrm{obs}}$. This yields the
	reconstructed field $\boldsymbol u^\ast$, provided the reconstructed
	source and boundary datum satisfy the required compatibility condition.
	
	The transport equation involved in the second stage is discretized
	using a SUPG finite-element method; see
	\cite{galeao1988consistent}. The complete procedure is summarized in
	Algorithm~\ref{alg:two-stage}.
	
	\subsection{Direct reconstruction}
	\label{subsec:direct}
	
	In the direct approach, the auxiliary source is eliminated using
	\[
	\boldsymbol z=\nabla\times(\boldsymbol u\times\boldsymbol F),
	\]
	and $\boldsymbol u$ is reconstructed in a single optimization step.
	We solve
	\begin{equation}
		\label{eq:direct-minimization}
		\boldsymbol u^\ast
		=
		\underset{\boldsymbol u\in\mathbb Z}{\operatorname{argmin}}
		\left\{
		\frac12
		\|\boldsymbol b(\boldsymbol u)(T)
		-\boldsymbol b_{\mathrm{obs}}(T)\|_{\mathbb H}^2
		+
		\frac{\gamma}{2}
		\|\boldsymbol u\|_{\mathbb X_{\boldsymbol F}}^2
		\right\},
	\end{equation}
	where $\boldsymbol b(\boldsymbol u)$ solves \eqref{b-system} with
	$\boldsymbol v=\alpha\boldsymbol u$, and
	$\boldsymbol b_{\mathrm{obs}}(T)$ is the measured terminal magnetic
	state. The admissible set is
	\[
	\mathbb Z
	=
	\left\{
	\boldsymbol u\in\mathbb X_{\boldsymbol F}:
	\nabla\cdot\boldsymbol u=0\text{ in }\Omega,\quad
	\gamma_-\boldsymbol u=\boldsymbol g_{\mathrm{obs}}
	\text{ on }\Gamma_{\mathrm{in}}
	\right\}.
	\]
	The minimization is understood over a nonempty admissible set.
	
	The two approaches differ in the variable optimized directly. The
	two-stage method first reconstructs $\boldsymbol z$ and then solves
	the curl--div problem for $\boldsymbol u$, whereas the direct method
	optimizes over $\boldsymbol u$ subject to the divergence and boundary
	constraints. Both methods ultimately produce a reconstruction of
	$\boldsymbol u$.
	
	\begin{algorithm}[H]
		\caption{Two-stage reconstruction of $\boldsymbol u$}
		\label{alg:two-stage}
		\begin{algorithmic}[1]
			\Require Initial state $\boldsymbol y_0$, corrected terminal observation
			$\boldsymbol y_{\mathrm{obs}}(T)$, background field $\boldsymbol F$,
			time profile $\alpha$, inflow datum $\boldsymbol g_{\mathrm{obs}}$,
			regularization parameter $\gamma>0$, tolerance $\varepsilon\geq0$.
			\Ensure Reconstructed velocity field $\boldsymbol u^\ast$.
			
			\State Compute
			$\|\mathcal C(\boldsymbol g_{\mathrm{obs}})\|_
			{L^2(\Gamma_{\mathrm{in}})}$.
			\If{$\|\mathcal C(\boldsymbol g_{\mathrm{obs}})\|_
				{L^2(\Gamma_{\mathrm{in}})}>\varepsilon$}
			\State Reject the boundary datum, or replace it by compatible data
			and use those data consistently in the subsequent steps.
			\EndIf
			
			\State Reconstruct $\boldsymbol z^\ast$ by solving
			\[
			\underset{\boldsymbol z\in\mathbb H}{\operatorname{argmin}}
			\left\{
			\frac12
			\|\boldsymbol y(\boldsymbol z)(T)
			-\boldsymbol y_{\mathrm{obs}}(T)\|_{\mathbb H}^2
			+
			\frac{\gamma}{2}\|\boldsymbol z\|_{\mathbb H}^2
			\right\},
			\]
			where $\boldsymbol y(\boldsymbol z)$ solves
			\eqref{eq:abstract-inverseh}.
			
			\State Check that $\boldsymbol z^\ast$ and the prescribed inflow data
			satisfy the compatibility condition required by
			\eqref{eq:div-curl}.
			
			\State Solve the curl--div problem~\eqref{eq:div-curl} with source
			$\boldsymbol z^\ast$ and inflow datum
			$\boldsymbol g_{\mathrm{obs}}$.
			
			\State \Return $\boldsymbol u^\ast$.
		\end{algorithmic}
	\end{algorithm}

	\subsection{Application to a manufactured solution}
	\label{subsec:manufactured}
	
	We test both reconstruction strategies on the three-dimensional unit
	ball
	\[
	\Omega
	=
	\left\{
	(x_1,x_2,x_3)\in\mathbb R^3:
	x_1^2+x_2^2+x_3^2<1
	\right\},
	\]
	whose outward unit normal is
	$\boldsymbol n=(x_1,x_2,x_3)^T$. We choose the constant background
	field
	\[
	\boldsymbol F=(1,0,0)^T.
	\]
	Thus $\boldsymbol F\in C^2(\overline\Omega)^3$,
	$\nabla\cdot\boldsymbol F=0$, and
	$\Delta\boldsymbol F=\boldsymbol 0$. We also take
	$\alpha\equiv1$. Consequently,
	\[
	\boldsymbol v=\boldsymbol u,\qquad
	\boldsymbol h=\boldsymbol z,\qquad
	\boldsymbol z=\nabla\times(\boldsymbol u\times\boldsymbol F).
	\]
	
	The flow generated by $\boldsymbol F$ is
	\[
	\boldsymbol X(t,\boldsymbol x_0)
	=\boldsymbol x_0+t(1,0,0).
	\]
	Since $\boldsymbol F\cdot\boldsymbol n=x_1$ on the unit sphere, the
	inflow boundary is
	\[
	\Gamma_{\mathrm{in}}
	=
	\left\{
	\boldsymbol x\in\partial\Omega:x_1<0
	\right\}.
	\]
	Every point in $\Omega$ lies on a unique characteristic segment
	entering through $\Gamma_{\mathrm{in}}$. The grazing set $x_1=0$ has
	zero surface measure and does not affect this almost-everywhere
	characteristic parametrization.
	
	We choose the exact velocity field
	\begin{equation}
		\label{u_true}
		\boldsymbol u_{\mathrm{true}}(x_1,x_2,x_3)
		=
		\begin{pmatrix}
			0\\
			x_1x_3\\
			-x_1x_2
		\end{pmatrix},
	\end{equation}
	and prescribe the inflow datum
	\[
	\boldsymbol g=\gamma_-\boldsymbol u_{\mathrm{true}}.
	\]
	
	We verify the properties of the manufactured source. First,
	\[
	\nabla\cdot\boldsymbol u_{\mathrm{true}}
	=
	\partial_2(x_1x_3)+\partial_3(-x_1x_2)=0.
	\]
	Moreover, on $\partial\Omega$,
	\[
	\boldsymbol u_{\mathrm{true}}\cdot\boldsymbol n
	=x_1x_2x_3-x_1x_2x_3=0.
	\]
	Thus the exact velocity is divergence-free and satisfies the
	no-penetration condition.
	
	Since $\boldsymbol F$ is constant and both fields are divergence-free,
	the vector identity for the curl of a cross product gives
	\[
	\boldsymbol z_{\mathrm{true}}
	=\nabla\times
	(\boldsymbol u_{\mathrm{true}}\times\boldsymbol F)
	=\partial_1\boldsymbol u_{\mathrm{true}}
	=
	\begin{pmatrix}
		0\\
		x_3\\
		-x_2
	\end{pmatrix}.
	\]
	It follows that
	\[
	\nabla\cdot\boldsymbol z_{\mathrm{true}}=0.
	\]
	Furthermore, on $\partial\Omega$,
	\[
	\boldsymbol z_{\mathrm{true}}\cdot\boldsymbol n
	=(0,x_3,-x_2)\cdot(x_1,x_2,x_3)
	=x_2x_3-x_2x_3=0.
	\]
	Hence $\boldsymbol z_{\mathrm{true}}\in\mathbb H$, provided
	$\mathbb H$ is the divergence-free space with zero normal trace used
	in the preceding analysis.
	
	The compatibility condition for the manufactured boundary datum
	holds because $\boldsymbol g$ is the inflow trace of the smooth field
	$\boldsymbol u_{\mathrm{true}}$, which satisfies the curl--div system
	with source $\boldsymbol z_{\mathrm{true}}$. Thus the manufactured
	data are compatible with \eqref{eq:div-curl}.

	\subsection{Implementation details}
	\label{subsec:forward-optimization}
	
	The computational mesh is generated using MATLAB's
	\texttt{multisphere(1)} geometry with mesh-size parameter $0.15$,
	yielding a tetrahedral mesh and continuous piecewise linear ($P_1$)
	vector finite elements. The implementation uses the polyhedral
	approximation of the unit ball.
	
	\paragraph{Weak formulation of the $\boldsymbol b$-system.}
	
	Let $\mathcal T_h$ denote the tetrahedral mesh. We use the finite-element
	space
	\[
	\mathbb V_h
	=
	\left\{
	\boldsymbol\varphi_h\in[P_1(\mathcal T_h)]^3:
	\boldsymbol\varphi_h\cdot\boldsymbol n=0
	\text{ on }\partial\Omega_h
	\right\},
	\]
	where the boundary condition is imposed with respect to the normal
	on the polyhedral boundary $\partial\Omega_h$.
	
	Using backward Euler with time step $\Delta t$, we seek
	$\boldsymbol b_h^n\in\mathbb V_h$ such that, for all
	$\boldsymbol\varphi_h\in\mathbb V_h$,
	\begin{equation}
		\label{eq:fe-b-weak}
		\left(
		\frac{\boldsymbol b_h^n-\boldsymbol b_h^{n-1}}{\Delta t},
		\boldsymbol\varphi_h
		\right)
		+
		\left(
		\nabla\times\boldsymbol b_h^n,
		\nabla\times\boldsymbol\varphi_h
		\right)
		=
		\left(
		\boldsymbol h_h^n,\boldsymbol\varphi_h
		\right).
	\end{equation}
	The discrete source is
	\begin{equation}
		\label{eq:h-b-source}
		\boldsymbol h_h^n
		=
		\Delta\boldsymbol F^n
		+
		\nabla\times(\boldsymbol v_h^n\times\boldsymbol F^n)
		=
		\Delta\boldsymbol F^n+\alpha^n\boldsymbol z_h^n,
	\end{equation}
	where
	\[
	\boldsymbol z_h^n
	=\nabla\times(\boldsymbol u_h^n\times\boldsymbol F^n).
	\]
	For the manufactured test,
	$\alpha\equiv1$ and $\Delta\boldsymbol F=\boldsymbol0$, so
	$\boldsymbol h_h^n=\boldsymbol z_h^n$.
	
	The curl--curl term gives the natural boundary condition
	$\boldsymbol n\times(\nabla\times\boldsymbol b)=\boldsymbol0$
	in the corresponding weak formulation, provided the solution has
	sufficient regularity for that boundary condition to be meaningful.
	The constraint
	$\boldsymbol b_h\cdot\boldsymbol n=0$ is imposed through the
	definition of $\mathbb V_h$.
	
	The curl--curl formulation alone does not enforce
	$\nabla\cdot\boldsymbol b_h=0$. If the implementation applies a
	discrete Helmholtz--Leray projection after each time step, that
	projection must be defined with respect to the discrete divergence
	operator and the imposed boundary conditions. The resulting projected
	scheme should be distinguished from the unprojected variational
	problem \eqref{eq:fe-b-weak}.
	
	\paragraph{Finite-element scheme.}
	
	Let $\boldsymbol M_v$ be the vector mass matrix, $\boldsymbol A_h$
	the curl--curl stiffness matrix, and $\boldsymbol H_h^n$ the load
	vector associated with $\boldsymbol h_h^n$. The backward Euler system
	corresponding to \eqref{eq:fe-b-weak} is
	\[
	(\boldsymbol M_v+\Delta t\,\boldsymbol A_h)\boldsymbol b^n
	=
	\boldsymbol M_v\boldsymbol b^{n-1}
	+\Delta t\,\boldsymbol H_h^n.
	\]
	If the code enforces the divergence constraint by projection, the
	computed field is then updated according to
	\[
	\boldsymbol b_h^n\leftarrow\mathbb P\boldsymbol b_h^n,
	\]
	where $\mathbb P$ denotes the discrete projection actually implemented.
	The definition of $\mathbb P$, including its treatment of the normal
	boundary condition, should be specified in the numerical method.
	A reference for projection methods is
	\cite{guermond1998approximation}; the precise projection used here
	must be consistent with the implementation.
	
	For the time discretization, we take
	\[
	T=1,\qquad N_t=100,\qquad
	\Delta t=0.01,\qquad \alpha(t)\equiv1.
	\]
	The regularization parameter is set to $\gamma=10^{-6}$.
	
	To assess sensitivity to measurement noise, we perturb the exact
	terminal observation according to
	\[
	\boldsymbol Y_{\mathrm{obs}}(T)
	=
	\boldsymbol Y_{\mathrm{obs}}^{\mathrm{exact}}(T)
	+\boldsymbol\eta,
	\]
	where the Gaussian noise realization is normalized so that
	\[
	\frac{\|\boldsymbol\eta\|_2}
	{\|\boldsymbol Y_{\mathrm{obs}}^{\mathrm{exact}}(T)\|_2}
	=\delta,
	\qquad
	\delta\in\{0,0.05,0.10,0.20\}.
	\]
	The same realization of the noisy terminal observation is supplied to
	both reconstruction methods at each noise level. The norm used in this
	normalization is the discrete Euclidean norm of the observation vector;
	it should not be confused with the $\mathbb H$-norm unless the
	discretization makes these norms equivalent under the adopted scaling.
	
	We measure the reconstruction error using the relative
	$\mathbb X_{\boldsymbol F}$-norm:
	\[
	E_u
	=
	\frac{
		\|\boldsymbol u_{\mathrm{rec}}-\boldsymbol u_{\mathrm{true}}\|_
		{\mathbb X_{\boldsymbol F}}
	}{
		\|\boldsymbol u_{\mathrm{true}}\|_{\mathbb X_{\boldsymbol F}}
	}.
	\]
	The reported values are specific to the mesh, time step, regularization
	parameter, noise realizations, and discrete norms used in the
	implementation. Although the same numerical value of $\gamma$ is used
	in both methods, the regularization terms involve different norms and
	therefore need not impose comparable regularization strengths.

	\subsection{Results and Discussion}
	\label{subsec:results}
	
	We compare the reconstructed velocity fields with the exact field
	$\boldsymbol u_{\mathrm{true}}$ in \eqref{u_true}. Table
	\ref{tab:error-comparison} reports the relative errors for the two
	methods at the four noise levels.
	
	\begin{table}[htbp]
		\caption{Relative reconstruction error $E_u$ for different terminal
			observation noise levels.}
		\label{tab:error-comparison}
		\centering
		\renewcommand{\arraystretch}{1.2}
		\begin{tabular}{lcc}
			\toprule
			\textbf{Noise level ($\delta$)}
			& \textbf{Two-stage method}
			& \textbf{Direct method}\\
			\midrule
			$0\%$  $(0.00)$ & $0.007$ & $0.015$\\
			$5\%$  $(0.05)$ & $0.011$ & $0.093$\\
			$10\%$ $(0.10)$ & $0.047$ & $0.245$\\
			$20\%$ $(0.20)$ & $0.078$ & $0.482$\\
			\bottomrule
		\end{tabular}
	\end{table}
	
	For the noiseless observation, both methods yield small reconstruction
	errors, with the two-stage method producing the smaller error in this
	test. As the noise level increases, the difference between the reported
	errors becomes more pronounced. At $20\%$ noise, the relative error is
	$0.078$ for the two-stage method and $0.482$ for the direct method.
	
	These results indicate that, for this manufactured example and the
	reported numerical settings, the two-stage method is less sensitive to
	the tested observation perturbations. In the two-stage approach, the
	auxiliary source is reconstructed through a regularized optimization
	problem before the curl--div system is solved. The second stage uses
	the transport structure of the stationary problem; in the
	implementation, the source is integrated along characteristics using
	a $12$-point Gauss--Legendre quadrature.
	
	The direct method instead reconstructs the velocity in a single
	optimization step, with the divergence and inflow constraints imposed
	simultaneously. The larger errors in the table suggest greater
	sensitivity to the tested noise realizations for this discretization
	and regularization choice. The observed errors alone, however, do not
	identify a unique cause for this difference. In particular, the two
	methods use different regularization norms, and their effective
	regularization strengths need not be equal even when the numerical
	parameter $\gamma$ is the same.

	\begin{figure}[h!]
		\centering
		\includegraphics[width=0.24\textwidth]{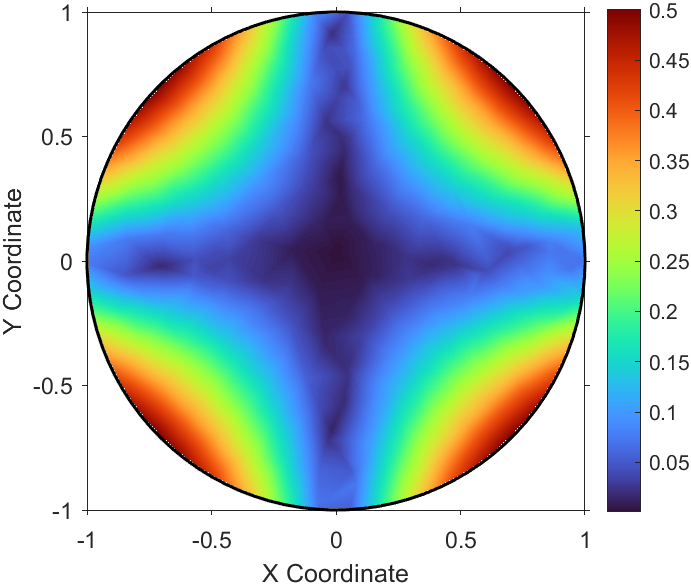}  
		\caption{Slice at $z=0$ of the exact field $\boldsymbol{u}_{\mathrm{true}}$
			defined in \eqref{u_true}.}
		\label{fig:utrue}
	\end{figure}
	
	\begin{figure}[h!]
		\centering
		\begin{subfigure}[b]{0.24\textwidth}
			\centering
			\includegraphics[width=\textwidth]{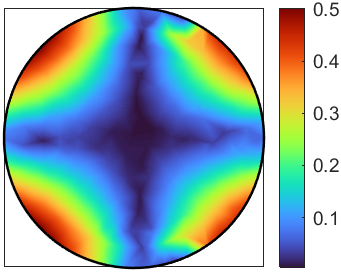}
			\caption{\(0\%\) noise}
			\label{fig:final-two-0}
		\end{subfigure}
		\hfill  
		\begin{subfigure}[b]{0.24\textwidth}
			\centering
			\includegraphics[width=\textwidth]{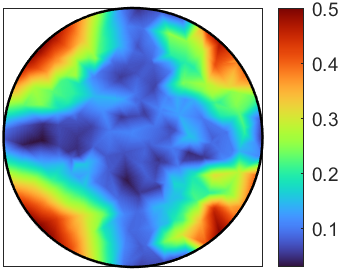}
			\caption{\(5\%\) noise}
			\label{fig:final-two-5}
		\end{subfigure}
		\hfill  
		\begin{subfigure}[b]{0.24\textwidth}
			\centering
			\includegraphics[width=\textwidth]{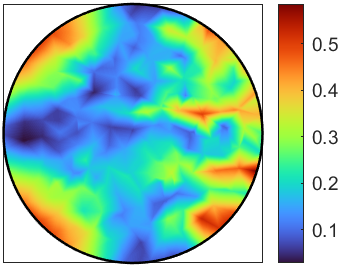}
			\caption{\(10\%\) noise}
			\label{fig:final-two-10}
		\end{subfigure}
		\hfill  
		\begin{subfigure}[b]{0.24\textwidth}
			\centering
			\includegraphics[width=\textwidth]{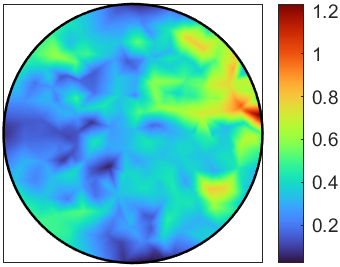}
			\caption{\(20\%\) noise}
			\label{fig:final-two-20}
		\end{subfigure}
		\caption{Slice at $z=0$ of the reconstructed field $\boldsymbol{u}$ obtained
			with the \textbf{two-stage method} from the terminal observation under
			varying noise levels. }
		\label{fig:final-two-stage}
	\end{figure}
	
	\begin{figure}[h!]
		\centering
		\begin{subfigure}[b]{0.24\textwidth}
			\centering
			\includegraphics[width=\textwidth]{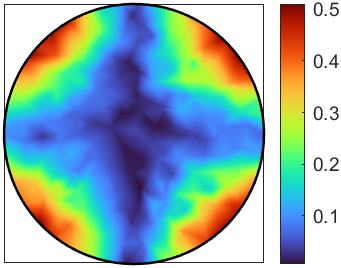}
			\caption{\(0\%\) noise}
			\label{fig:final-dir-0}
		\end{subfigure}
		\hfill  
		\begin{subfigure}[b]{0.24\textwidth}
			\centering
			\includegraphics[width=\textwidth]{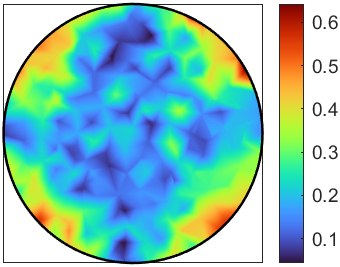}
			\caption{\(5\%\) noise}
			\label{fig:final-dir-5}
		\end{subfigure}
		\hfill  
		\begin{subfigure}[b]{0.24\textwidth}
			\centering
			\includegraphics[width=\textwidth]{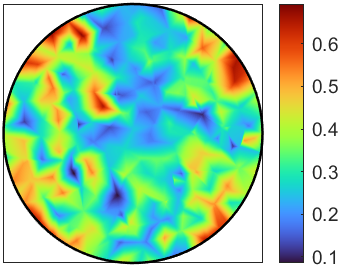}
			\caption{\(10\%\) noise}
			\label{fig:final-dir-10}
		\end{subfigure}
		\hfill  
		\begin{subfigure}[b]{0.24\textwidth}
			\centering
			\includegraphics[width=\textwidth]{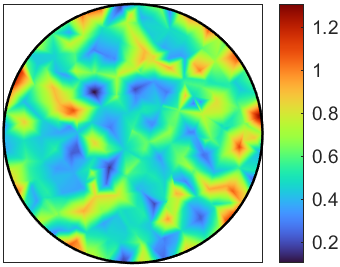}
			\caption{\(20\%\) noise}
			\label{fig:final-dir-20}
		\end{subfigure}
		\caption{Slice at $z=0$ of the reconstructed field $\boldsymbol{u}$ obtained
			with the \textbf{direct method} from the same observation under varying
			noise levels.}
		\label{fig:final-direct}
	\end{figure}
	
	
	
	
	\section{Appendix}
	
	\subsection*{Proof of Proposition~\protect\ref{Proposition2}}
	
	\label{AppendixB: proof curl-div}
	
	\begin{proof}
		We split the argument into four steps.
		
		\paragraph{Step 1: Reduction to a linear transport system.}
		By using \eqref{First-order} and since
		$[D\boldsymbol{F}\,\boldsymbol{u}]_{i}
		=\sum_{j=1}^{3}(\partial_{j}F_{i})u_{j}
		=(\boldsymbol{u}\cdot\nabla)F_{i}$, we have
		$(\boldsymbol{u}\cdot\nabla)\boldsymbol{F}=D\boldsymbol{F}\,\boldsymbol{u}$.
		Likewise $(\boldsymbol{F}\cdot\nabla)\boldsymbol{u}=D\boldsymbol{u}\,\boldsymbol{F}$. Hence
		\eqref{eq:div-curl} is equivalent to the first-order linear transport
		system with zero-order coupling
		\begin{equation}
			\label{eq:transport}
			\begin{cases}
				D\boldsymbol{u}\,\boldsymbol{F}-D\boldsymbol{F}\,\boldsymbol{u}=\boldsymbol{z}
				& \text{in }\Omega,\\[2pt]
				\boldsymbol{u}=\boldsymbol{g} & \text{on }\Gamma_{\mathrm{in}},
			\end{cases}
		\end{equation}
		together with the divergence constraint $\nabla\cdot\boldsymbol{u}=0$ in $\Omega$.
		
		\paragraph{Step 2: Existence and energy estimate.}
		Because $\boldsymbol{F}\in[C^{2}(\overline{\Omega})]^{3}$, the flow
		$\boldsymbol{X}$ is of class $C^{1}$. By the \emph{$\Omega$-filling} assumption according to Definition \ref{Definition-Omega-Filing} and the injectivity of the flow (by Lemma \ref{lem:flow-injective}), for
		a.e.\ $\boldsymbol{x}\in\Omega$ there exists a unique pair
		$(t,\boldsymbol{x}_{0})\in[0,T_{\mathrm{fill}}]\times\Gamma_{\mathrm{in}}$ such
		that $\boldsymbol{x}=\boldsymbol{X}(t,\boldsymbol{x}_{0})$. Let
		$\tau(\boldsymbol{x}_{0})\in(0,T_{\mathrm{fill}}]$ denote the exit time of the
		forward characteristic starting from $\boldsymbol{x}_{0}$. 
		
		For $\boldsymbol{x}_{0}\in\Gamma_{\mathrm{in}}$ let
		$\boldsymbol{R}(t,s;\boldsymbol{x}_{0})$ be the resolvent of the linear
		variational equation along the characteristic through $\boldsymbol{x}_{0}$,
		defined as the unique $C^{1}$ solution of
		\begin{equation}
			\label{eq:resolvent}
			\partial_t\boldsymbol{R}(t,s;\boldsymbol{x}_{0})
			=D\boldsymbol{F}(\boldsymbol{X}(t,\boldsymbol{x}_{0}))\,
			\boldsymbol{R}(t,s;\boldsymbol{x}_{0}),\qquad
			\boldsymbol{R}(s,s;\boldsymbol{x}_{0})=\boldsymbol{I},
		\end{equation}
		for $0\le s,t\le\tau(\boldsymbol{x}_{0})$. Since
		$\boldsymbol{F}\in[C^{2}(\overline{\Omega})]^{3}$, the coefficient
		$D\boldsymbol{F}\circ\boldsymbol{X}$ is continuous and bounded; therefore
		\eqref{eq:resolvent} has a unique $C^{1}$ solution satisfying the uniform bound
		\begin{equation}
			\label{eq:R-bounds}
			\|\boldsymbol{R}(t,s;\boldsymbol{x}_{0})\|
			\le e^{\|D\boldsymbol{F}\|_{\infty}|t-s|},\qquad
			0\le s,t\le\tau(\boldsymbol{x}_{0}),\;\boldsymbol{x}_{0}\in\Gamma_{\mathrm{in}}.
		\end{equation}
		
		Since the space
		$[C_c^{\infty}(\Omega)]^{3}\cap\{\nabla\cdot=0\}$ is dense in $\mathbb{H}$
		\cite[Theorem 1.4]{temam1977navier} and
		$[C^{\infty}(\Gamma_{\mathrm{in}})]^{3}$ is dense in
		$L^{2}(\Gamma_{\mathrm{in}};\tau|\boldsymbol{F}\cdot\boldsymbol{n}|\,\mathrm{d}\sigma)$,
		for smooth data
		$\boldsymbol{z}\in[C_c^{\infty}(\Omega)]^{3}\cap\{\nabla\cdot\boldsymbol{z}=0\}$
		and $\boldsymbol{g}\in[C^{\infty}(\Gamma_{\mathrm{in}})]^{3}$, we define
		$\boldsymbol{u}$ along characteristics by Duhamel's formula:
		\begin{equation}
			\label{eq:duhamel}
			\boldsymbol{u}(\boldsymbol{X}(t,\boldsymbol{x}_{0}))
			:=\boldsymbol{R}(t,0;\boldsymbol{x}_{0})\,\boldsymbol{g}(\boldsymbol{x}_{0})
			+\int_{0}^{t}\boldsymbol{R}(t,s;\boldsymbol{x}_{0})
			\boldsymbol{z}(\boldsymbol{X}(s,\boldsymbol{x}_{0}))\,\mathrm{d}s,
		\end{equation}
		for $\boldsymbol{x}_{0}\in\Gamma_{\mathrm{in}}$ and
		$0\le t\le\tau(\boldsymbol{x}_{0})$. Differentiating \eqref{eq:duhamel} with
		respect to $t$ and using \eqref{eq:resolvent}, we obtain
		\[
		\partial_t\boldsymbol{u}(\boldsymbol{X}(t,\boldsymbol{x}_{0}))
		=D\boldsymbol{F}(\boldsymbol{X}(t,\boldsymbol{x}_{0}))\,
		\boldsymbol{u}(\boldsymbol{X}(t,\boldsymbol{x}_{0}))
		+\boldsymbol{z}(\boldsymbol{X}(t,\boldsymbol{x}_{0})).
		\]
		Since
		$\partial_t\boldsymbol{u}(\boldsymbol{X}(t,\boldsymbol{x}_{0}))
		=D\boldsymbol{u}\,\boldsymbol{F}(\boldsymbol{X}(t,\boldsymbol{x}_{0}))$, this
		shows that $\boldsymbol{u}$ satisfies the transport equation in
		\eqref{eq:transport}. Moreover, setting $t=0$ in \eqref{eq:duhamel} gives
		$\boldsymbol{u}(\boldsymbol{x}_{0})=\boldsymbol{g}(\boldsymbol{x}_{0})$, so the
		boundary condition is satisfied.
		
		To establish the stability estimate, let $\boldsymbol{x}_{0}\in\Gamma_{\mathrm{in}}$
		and $t\in[0,\tau(\boldsymbol{x}_{0})]$. From \eqref{eq:duhamel} and
		\eqref{eq:R-bounds},
		\[
		|\boldsymbol{u}(\boldsymbol{X}(t,\boldsymbol{x}_{0}))|
		\le e^{\|D\boldsymbol{F}\|_{\infty}T_{\mathrm{fill}}}|\boldsymbol{g}(\boldsymbol{x}_{0})|
		+\int_{0}^{t}e^{\|D\boldsymbol{F}\|_{\infty}(t-s)}
		|\boldsymbol{z}(\boldsymbol{X}(s,\boldsymbol{x}_{0}))|\,\mathrm{d}s.
		\]
		Squaring, using $(a+b)^{2}\le 2a^{2}+2b^{2}$ and Cauchy--Schwarz,
		\[
		|\boldsymbol{u}(\boldsymbol{X}(t,\boldsymbol{x}_{0}))|^{2}
		\le 2e^{2\|D\boldsymbol{F}\|_{\infty}T_{\mathrm{fill}}}|\boldsymbol{g}(\boldsymbol{x}_{0})|^{2}
		+2T_{\mathrm{fill}}\,e^{2\|D\boldsymbol{F}\|_{\infty}T_{\mathrm{fill}}}
		\int_{0}^{t}|\boldsymbol{z}(\boldsymbol{X}(s,\boldsymbol{x}_{0}))|^{2}\,
		\mathrm{d}s.
		\]
		Integrating this inequality with respect to $t$ from $0$ to
		$\tau(\boldsymbol{x}_{0})$, and keeping the factor $\tau(\boldsymbol{x}_{0})$
		produced by the first term, we
		obtain
		\[
		\int_{0}^{\tau(\boldsymbol{x}_{0})}
		|\boldsymbol{u}(\boldsymbol{X}(t,\boldsymbol{x}_{0}))|^{2}\,\mathrm{d}t
		\le 2e^{2\|D\boldsymbol{F}\|_{\infty}T_{\mathrm{fill}}}\,
		\tau(\boldsymbol{x}_{0})\,|\boldsymbol{g}(\boldsymbol{x}_{0})|^{2}
		+2T_{\mathrm{fill}}^{2}e^{2\|D\boldsymbol{F}\|_{\infty}T_{\mathrm{fill}}}
		\int_{0}^{\tau(\boldsymbol{x}_{0})}
		|\boldsymbol{z}(\boldsymbol{X}(s,\boldsymbol{x}_{0}))|^{2}\,
		\mathrm{d}s.
		\]
		Multiplying by $|\boldsymbol{F}(\boldsymbol{x}_{0})\cdot\boldsymbol{n}(\boldsymbol{x}_{0})|$
		and integrating over $\Gamma_{\mathrm{in}}$, the change of variables formula
		\eqref{eq:change-of-variables} yields
		\begin{equation}
			\label{eq:L2-estimate}
			\|\boldsymbol{u}\|_{\mathbb{H}}^2
			\le C\bigl(\|\boldsymbol{z}\|_{\mathbb{H}}^2
			+\|\boldsymbol{g}\|_{L^{2}(\Gamma_{\mathrm{in}};\tau|\boldsymbol{F}\cdot\boldsymbol{n}|\,\mathrm{d}\sigma)}^2\bigr),
		\end{equation}
		with $C=C(T_{\mathrm{fill}},\|D\boldsymbol{F}\|_{\infty})>0$. From the
		transport equation \eqref{eq:transport} we have
		$D\boldsymbol{u}\,\boldsymbol{F}=D\boldsymbol{F}\,\boldsymbol{u}+\boldsymbol{z}$,
		whence
		\[
		\|D\boldsymbol{u}\,\boldsymbol{F}\|_{\mathbb{H}}
		\le\|D\boldsymbol{F}\|_{L^{\infty}}\|\boldsymbol{u}\|_{\mathbb{H}}
		+\|\boldsymbol{z}\|_{\mathbb{H}}.
		\]
		Squaring and combining with \eqref{eq:L2-estimate} gives the estimate
		\begin{equation}
			\label{eq:smooth-estimate}
			\|\boldsymbol{u}\|_{\mathbb{H}}^{2}+\|D\boldsymbol{u}\,\boldsymbol{F}\|_{\mathbb{H}}^{2}
			\le C\bigl(\|\boldsymbol{z}\|_{\mathbb{H}}^2
			+\|\boldsymbol{g}\|_{L^{2}(\Gamma_{\mathrm{in}};\tau|\boldsymbol{F}\cdot\boldsymbol{n}|\,\mathrm{d}\sigma)}^2\bigr),
		\end{equation}
		for smooth data.
		\medskip
		
		\noindent\textbf{Density argument: extension to general data.}
		The estimate \eqref{eq:smooth-estimate} was proved for smooth data
		$\boldsymbol{z}\in[C_c^{\infty}(\Omega)]^{3}\cap\{\nabla\cdot =0\}$, and
		$\boldsymbol{g}\in[C^{\infty}(\Gamma_{\mathrm{in}})]^{3}$. We now extend it to
		general
		\[
		\boldsymbol{z}\in\mathbb{H},
		\qquad
		\boldsymbol{g}\in L^{2}(\Gamma_{\mathrm{in}};\tau|\boldsymbol{F}\cdot\boldsymbol{n}|\,\mathrm{d}\sigma).
		\]
		Choose sequences
		$\boldsymbol{z}_{n}\in	\boldsymbol{z}\in[C_c^{\infty}(\Omega)]^{3}\cap\{\nabla\cdot =0\}$,
		$\boldsymbol{g}_{n}\in[C^{\infty}(\Gamma_{\mathrm{in}})]^{3}$ with
		$\boldsymbol{z}_{n}\to\boldsymbol{z}$ in $\mathbb{H}$,
		$\boldsymbol{g}_{n}\to\boldsymbol{g}$ in
		$L^{2}(\Gamma_{\mathrm{in}};\tau|\boldsymbol{F}\cdot\boldsymbol{n}|\,\mathrm{d}\sigma)$.
		For each $n$, let $\boldsymbol{u}_{n}$ be the smooth solution from
		\eqref{eq:duhamel}. Applying \eqref{eq:smooth-estimate} to
		$\boldsymbol{u}_{n}-\boldsymbol{u}_{m}$,
		\[
		\|\boldsymbol{u}_{n}-\boldsymbol{u}_{m}\|_{\mathbb{H}}^{2}
		+\|D(\boldsymbol{u}_{n}-\boldsymbol{u}_{m})\,\boldsymbol{F}\|_{\mathbb{H}}^{2}
		\le
		C\bigl(\|\boldsymbol{z}_{n}-\boldsymbol{z}_{m}\|_{\mathbb{H}}^{2}
		+\|\boldsymbol{g}_{n}-\boldsymbol{g}_{m}\|_{L^{2}(\Gamma_{\mathrm{in}};\tau|\boldsymbol{F}\cdot\boldsymbol{n}|\,\mathrm{d}\sigma)}^{2}\bigr).
		\]
		Hence $(\boldsymbol{u}_{n})$ and $(D\boldsymbol{u}_{n}\,\boldsymbol{F})$ are
		Cauchy in $[L^{2}(\Omega)]^{3}$; let $\boldsymbol{u}$ be the limit of
		$(\boldsymbol{u}_{n})$. Passing to the limit in the transport equation (tested
		against $\varphi\in C^{\infty}_{c}(\Omega)$) gives
		$D\boldsymbol{u}\,\boldsymbol{F}-D\boldsymbol{F}\,\boldsymbol{u}=\boldsymbol{z}$
		in $\mathbb{H}$. The boundary condition is preserved by the continuity of the
		trace operator $\gamma_{-}$ in the weighted space
		$L^{2}(\Gamma_{\mathrm{in}};\tau|\boldsymbol{F}\cdot\boldsymbol{n}|\,\mathrm{d}\sigma)$
		(Proposition~\ref{proposition-trace}): since the $\boldsymbol{u}_{n}$ are not
		divergence free, we use it in the form given by its proof, which only requires
		$\boldsymbol{u}\in[L^{2}(\Omega)]^{3}$ and
		$D\boldsymbol{u}\,\boldsymbol{F}\in[L^{2}(\Omega)]^{3}$:
		\begin{equation}
			\label{eq:trace-graph}
			\int_{\Gamma_{\mathrm{in}}}|\gamma_{-}\boldsymbol{u}|^{2}\,\tau\,
			|\boldsymbol{F}\cdot\boldsymbol{n}|\,\mathrm{d}\sigma
			\le 2\|\boldsymbol{u}\|_{\mathbb{H}}^{2}
			+T_{\mathrm{fill}}^{2}\|D\boldsymbol{u}\,\boldsymbol{F}\|_{\mathbb{H}}^{2}.
		\end{equation}
		Applying \eqref{eq:trace-graph} to $\boldsymbol{u}_{n}-\boldsymbol{u}$ and
		using $\gamma_{-}\boldsymbol{u}_{n}=\boldsymbol{g}_{n}\to\boldsymbol{g}$, we
		get $\gamma_{-}\boldsymbol{u}=\boldsymbol{g}$. Passing to the limit in
		\eqref{eq:smooth-estimate},
		\begin{equation}
			\label{eq:general-estimate}
			\|\boldsymbol{u}\|_{\mathbb{H}}^{2}
			+\|D\boldsymbol{u}\,\boldsymbol{F}\|_{\mathbb{H}}^{2}
			\le
			C\bigl(\|\boldsymbol{z}\|_{\mathbb{H}}^{2}
			+\|\boldsymbol{g}\|_{L^{2}(\Gamma_{\mathrm{in}};\tau|\boldsymbol{F}\cdot\boldsymbol{n}|\,\mathrm{d}\sigma)}^{2}\bigr).
		\end{equation}
		\paragraph{Step 3: Necessity and sufficiency of the compatibility condition.}
		\bigskip
		\textbf{Necessity.}	Assume $\boldsymbol{u}\in\mathbb{X}_{\boldsymbol{F}}$ solves
		\eqref{eq:div-curl}. By Step~1 it solves \eqref{eq:transport} with
		$\nabla\cdot\boldsymbol{u}=0$.
		
		Approximate $(\boldsymbol{z},\boldsymbol{g})$ by smooth data
		$(\boldsymbol{z}_{n},\boldsymbol{g}_{n})$ with
		$\boldsymbol{z}_{n}\in[C^{\infty}(\overline{\Omega})]^{3}\cap\mathbb{H}$,
		$\boldsymbol{g}_{n}\in[C^{\infty}(\Gamma_{\mathrm{in}})]^{3}$,
		$\boldsymbol{z}_{n}\to\boldsymbol{z}$ in $\mathbb{H}$,
		$\boldsymbol{g}_{n}\to\boldsymbol{g}$ in $[H^{1}(\Gamma_{\mathrm{in}})]^{3}$.
		Let $\boldsymbol{u}_{n}$ be the smooth solutions from \eqref{eq:duhamel}. By
		the density used above,
		$\|\boldsymbol{u}_{n}-\boldsymbol{u}\|_{\mathbb{H}}
		+\|D(\boldsymbol{u}_{n}-\boldsymbol{u})\,\boldsymbol{F}\|_{\mathbb{H}}\to 0$,
		and $\gamma_{-}\boldsymbol{u}_{n}=\boldsymbol{g}_{n}\to\boldsymbol{g}$ in
		$L^{2}(\Gamma_{\mathrm{in}};\tau|\boldsymbol{F}\cdot\boldsymbol{n}|\,\mathrm{d}\sigma)$.
		
		For smooth $\boldsymbol{u}_{n}$, decompose $\boldsymbol{F}$ into its tangential
		and normal components:
		\[
		\boldsymbol{F}_{\Gamma}=\boldsymbol{F}-(\boldsymbol{F}\cdot\boldsymbol{n})\boldsymbol{n}.
		\]
		On $\Gamma_{\mathrm{in}}$,
		$D\boldsymbol{u}_{n}\,\boldsymbol{F}
		=D_{\Gamma}\boldsymbol{u}_{n}\,\boldsymbol{F}_{\Gamma}
		+(\boldsymbol{F}\cdot\boldsymbol{n})\,\partial_{\boldsymbol{n}}\boldsymbol{u}_{n}$.
		Taking the scalar product with $\boldsymbol{n}$ and using
		$\boldsymbol{u}_{n}=\boldsymbol{g}_{n}$ on $\Gamma_{\mathrm{in}}$,
		\begin{equation}
			\label{eq:necessity-convective}
			D\boldsymbol{u}_{n}\,\boldsymbol{F}\cdot\boldsymbol{n}
			=\bigl(D_{\Gamma}\boldsymbol{g}_{n}\,\boldsymbol{F}_{\Gamma}\bigr)\cdot\boldsymbol{n}
			+(\boldsymbol{F}\cdot\boldsymbol{n})\bigl(\partial_{\boldsymbol{n}}\boldsymbol{u}_{n}\cdot\boldsymbol{n}\bigr).
		\end{equation}
		From the transport equation \eqref{eq:transport}, we have $D\boldsymbol{u}_{n}\,\boldsymbol{F}-D\boldsymbol{F}\,\boldsymbol{u}_{n}=\boldsymbol{z}_{n}$
		and $\boldsymbol{z}_{n}\cdot\boldsymbol{n}=0$ on $\partial\Omega$,
		\begin{equation}
			\label{eq:necessity-transport}
			D\boldsymbol{u}_{n}\,\boldsymbol{F}\cdot\boldsymbol{n}
			=\bigl(D\boldsymbol{F}\,\boldsymbol{g}_{n}\bigr)\cdot\boldsymbol{n}.
		\end{equation}
		From $\nabla\cdot\boldsymbol{u}_{n}=0$ on the boundary, using \cite[formula (5.19), p.~522]{delfour2011shapes} gives 
		\begin{equation}
			\label{eq:normal-derivative}
			\partial_{\boldsymbol{n}}\boldsymbol{u}_{n}\cdot\boldsymbol{n}
			=-\nabla_{\Gamma_{\mathrm{in}}}\cdot\boldsymbol{g}_{n}.
		\end{equation}
		Substituting \eqref{eq:normal-derivative} into \eqref{eq:necessity-convective}
		and comparing with \eqref{eq:necessity-transport},
		\[
		(\boldsymbol{F}\cdot\boldsymbol{n})\nabla_{\Gamma_{\mathrm{in}}}\cdot\boldsymbol{g}_{n}
		+\bigl(D\boldsymbol{F}\,\boldsymbol{g}_{n}-D_{\Gamma}\boldsymbol{g}_{n}\,\boldsymbol{F}_{\Gamma}\bigr)\cdot\boldsymbol{n}
		=0,
		\]
		i.e.\ $C(\boldsymbol{g}_{n})=0$. Since $C$ is a first-order tangential operator
		with bounded coefficients, $\boldsymbol{g}_{n}\to\boldsymbol{g}$ in
		$[H^{1}(\Gamma_{\mathrm{in}})]^{3}$ implies
		$C(\boldsymbol{g}_{n})\to C(\boldsymbol{g})$ in $L^{2}(\Gamma_{\mathrm{in}})$.
		Hence $C(\boldsymbol{g})=0$ a.e.\ on $\Gamma_{\mathrm{in}}$.
		
		\medskip\noindent
		\textbf{Sufficiency.}
		We assume $C(\boldsymbol{g})=0$ a.e.\ on $\Gamma_{\mathrm{in}}$. Let
		$\boldsymbol{u}$ be the solution of \eqref{eq:transport}. Set $q:=\nabla\cdot\boldsymbol{u}$.
		
		For smooth data, $q$ is a continuous function and the computation below is
		classical. Since $\boldsymbol{z}=D\boldsymbol{u}\,\boldsymbol{F}-D\boldsymbol{F}\,\boldsymbol{u}$
		and $\nabla\cdot\boldsymbol{z}=0$, taking the divergence gives
		$\nabla\cdot(D\boldsymbol{u}\,\boldsymbol{F})-\nabla\cdot(D\boldsymbol{F}\,\boldsymbol{u})=0$.
		Using the identities
		\[
		\nabla\cdot(D\boldsymbol{u}\,\boldsymbol{F})
		=(\boldsymbol{F}\cdot\nabla)q+\sum_{i,j}(\partial_{i}F_{j})(\partial_{j}u_{i}),
		\]
		\[
		\nabla\cdot(D\boldsymbol{F}\,\boldsymbol{u})
		=\boldsymbol{u}\cdot\nabla(\nabla\cdot\boldsymbol{F})
		+\sum_{i,j}(\partial_{j}F_{i})(\partial_{i}u_{j}),
		\]
		and $\nabla\cdot\boldsymbol{F}=0$, the two sums cancel after swapping
		$i\leftrightarrow j$, giving $(\boldsymbol{F}\cdot\nabla)q=0$ in
		$\mathcal{D}'(\Omega)$.
		
		On $\Gamma_{\mathrm{in}}$, the tangential divergence identity
		\cite[formula (5.19), p.~522]{delfour2011shapes} applied to $\boldsymbol{u}$
		gives
		\[
		q=\nabla_{\Gamma_{\mathrm{in}}}\cdot\boldsymbol{g}
		+\partial_{\boldsymbol{n}}\boldsymbol{u}\cdot\boldsymbol{n}.
		\]
		Combined with the boundary decomposition of the transport equation and
		$\boldsymbol{z}\cdot\boldsymbol{n}=0$ on $\partial\Omega$,
		\[
		(\boldsymbol{F}\cdot\boldsymbol{n})\,q|_{\Gamma_{\mathrm{in}}}
		=(\boldsymbol{F}\cdot\boldsymbol{n})\nabla_{\Gamma_{\mathrm{in}}}\cdot\boldsymbol{g}
		+\bigl(D\boldsymbol{F}\,\boldsymbol{g}-D_{\Gamma}\boldsymbol{g}\,\boldsymbol{F}_{\Gamma}\bigr)\cdot\boldsymbol{n}
		=C(\boldsymbol{g}).
		\]
		Since $C(\boldsymbol{g})=0$, we have $q|_{\Gamma_{\mathrm{in}}}=0$ wherever
		$\boldsymbol{F}\cdot\boldsymbol{n}\neq0$. The scalar transport problem
		$(\boldsymbol{F}\cdot\nabla)q=0$ in $\Omega$, $q|_{\Gamma_{\mathrm{in}}}=0$,
		together with the \emph{$\Omega$-filling} property, gives $q\equiv0$. Hence
		$\nabla\cdot\boldsymbol{u}=0$.
		
		For general data, the same conclusion follows by density: approximate
		$(\boldsymbol{z},\boldsymbol{g})$ by smooth data with $C(\boldsymbol{g}_{n})\to0$,
		apply the smooth result to $\boldsymbol{u}_{n}$, and pass to the limit using
		the estimate \eqref{eq:general-estimate} (which controls $q_{n}$ in terms of
		$C(\boldsymbol{g}_{n})$ and $\boldsymbol{z}_{n}$).
		
		Combining necessity and sufficiency, $C(\boldsymbol{g})=0$ is necessary and
		sufficient for the solution of the transport problem to be divergence-free.
		
		\paragraph{Step 4: Uniqueness.}
		Let $\boldsymbol{u}_{1}$ and $\boldsymbol{u}_{2}$ be two solutions of
		\eqref{eq:div-curl} with the same data $(\boldsymbol{z},\boldsymbol{g})$. By
		Step~1 both solve \eqref{eq:transport}. Their difference
		$\boldsymbol{w}:=\boldsymbol{u}_{1}-\boldsymbol{u}_{2}$ solves
		\eqref{eq:transport} with $\boldsymbol{z}=0$ and $\boldsymbol{g}=0$. Applying
		the estimate \eqref{eq:general-estimate} to $\boldsymbol{w}$,
		\[
		\|\boldsymbol{w}\|_{\mathbb{H}}^{2}+\|D\boldsymbol{w}\,\boldsymbol{F}\|_{\mathbb{H}}^{2}
		\le C\bigl(0+0\bigr)=0.
		\]
		Hence $\boldsymbol{w}=0$ in $\Omega$. The solution is unique.
	\end{proof}
	
	\begin{lemma}
		\label{lem:change-of-variables} Let $\Omega\subset\mathbb{R}^{d}$ have $%
		C^{1} $ boundary and let $\boldsymbol{F}\in[C^{2}(\overline{\Omega})]^{d}$
		with $\nabla\cdot\boldsymbol{F}=0$ in $\Omega$. Assume that the flow $(t, 
		\boldsymbol{x}_{0})\mapsto\boldsymbol{X}(t,\boldsymbol{x}_{0})$ generated by 
		$\boldsymbol{F}$ is \emph{$\Omega$-filling} in the sense of Definition~\ref%
		{Definition-Omega-Filing}. Then the map 
		\begin{equation*}
			(t,\boldsymbol{x}_{0})\longmapsto\boldsymbol{X}(t,\boldsymbol{x}_{0}),
			\qquad (t,\boldsymbol{x}_{0})\in\mathcal{D} = \bigl\{(t,\boldsymbol{x}_{0}): 
			\boldsymbol{x}_{0}\in\Gamma_{\mathrm{in}},\ 0<t<\tau(\boldsymbol{x}_{0}) %
			\bigr\},
		\end{equation*}
		is a $C^{1}$-diffeomorphism onto $\Omega\setminus\mathcal{N}$ for some $%
		\mathcal{N}\subset\Omega$ with $|\mathcal{N}|=0$, and for every $f\in
		L^{1}(\Omega)$, 
		\begin{equation}  \label{eq:change-of-variables}
			\int_{\Omega}f(\boldsymbol{x})\,\mathrm{d}\boldsymbol{x} = \int_{\Gamma_{ 
					\mathrm{in}}}\int_{0}^{\tau(\boldsymbol{x}_{0})} f(\boldsymbol{X}(t, 
			\boldsymbol{x}_{0}))\, \bigl|\boldsymbol{F}(\boldsymbol{x}_{0})\cdot 
			\boldsymbol{n}(\boldsymbol{x}_{0})\bigr|\, \mathrm{d}t\,\mathrm{d}\sigma( 
			\boldsymbol{x}_{0}).
		\end{equation}
	\end{lemma}
	
	\begin{proof}
		Since the flow $\boldsymbol{X}$ is \emph{$\Omega$-filling}, and it is injective by Lemma \ref{lem:flow-injective}, the flow map
		$\boldsymbol{X}$ is a bijection from $\mathcal{D}$ onto
		$\Omega\setminus\mathcal{N}$ for some $\mathcal{N}\subset\Omega$
		with $|\mathcal{N}|=0$, and
		$\tau(\boldsymbol{x}_{0})\le T_{\mathrm{fill}}$ for
		$\sigma$-a.e.\ $\boldsymbol{x}_{0}\in\Gamma_{\mathrm{in}}$.
		
		Fix a local $C^{1}$ parametrization
		$\boldsymbol{x}_{0}=\boldsymbol{\gamma}(s)$,
		$s\in U\subset\mathbb{R}^{d-1}$, of $\Gamma_{\mathrm{in}}$, and set
		\[
		\boldsymbol{W}(t,s)
		:=
		\nabla_{(t,s)}\boldsymbol{X}(t,\boldsymbol{\gamma}(s))
		=
		\begin{bmatrix}
			\partial_{t}\boldsymbol{X} &
			\partial_{s_{1}}\boldsymbol{X} &
			\cdots &
			\partial_{s_{d-1}}\boldsymbol{X}
		\end{bmatrix},
		\]
		where the derivatives are evaluated at $(t,\boldsymbol{\gamma}(s))$.
		The first column is
		$\partial_{t}\boldsymbol{X}=\boldsymbol{F}(\boldsymbol{X})$.
		Differentiating the flow equation with respect to $s_{i}$ gives
		\[
		\partial_t\partial_{s_{i}}\boldsymbol{X}
		=
		D\boldsymbol{F}(\boldsymbol{X})\,\partial_{s_{i}}\boldsymbol{X},
		\qquad
		i=1,\dots,d-1,
		\]
		and similarly
		\[
		\partial_t\boldsymbol{F}(\boldsymbol{X})
		=
		D\boldsymbol{F}(\boldsymbol{X})\,\boldsymbol{F}(\boldsymbol{X}).
		\]
		Thus every column of $\boldsymbol{W}(t,s)$ satisfies the linear system
		\[
		\boldsymbol{Y}'(t)
		=
		D\boldsymbol{F}(\boldsymbol{X}(t,\boldsymbol{\gamma}(s)))\,\boldsymbol{Y}(t).
		\]
		By Liouville's formula,
		\[
		\partial_t\det\boldsymbol{W}(t,s)
		=
		\operatorname{tr}\!\bigl(D\boldsymbol{F}(\boldsymbol{X}(t,\boldsymbol{\gamma}(s)))\bigr)\,
		\det\boldsymbol{W}(t,s)
		=
		(\nabla\cdot\boldsymbol{F})(\boldsymbol{X}(t,\boldsymbol{\gamma}(s)))\,
		\det\boldsymbol{W}(t,s).
		\]
		Since $\nabla\cdot\boldsymbol{F}=0$ in $\Omega$, by Liouville's Theorem \cite[Proposition 7.1]{majda2006nonlinear}
		\[
		\det\boldsymbol{W}(t,s)=\det\boldsymbol{W}(0,s)
		\qquad\text{for all }t.
		\]
		
		At $t=0$,
		\[
		\boldsymbol{W}(0,s)=
		\begin{bmatrix}
			\boldsymbol{F}(\boldsymbol{\gamma}(s)) &
			\partial_{s_{1}}\boldsymbol{\gamma}(s) &
			\cdots &
			\partial_{s_{d-1}}\boldsymbol{\gamma}(s)
		\end{bmatrix}.
		\]
		Decompose
		$\boldsymbol{F}=\boldsymbol{F}_{\tau}+(\boldsymbol{F}\cdot\boldsymbol{n})\,\boldsymbol{n}$,
		where $\boldsymbol{F}_{\tau}$ is tangent to $\Gamma_{\mathrm{in}}$ at
		$\boldsymbol{\gamma}(s)$.
		By multilinearity of the determinant and the fact that the tangential vectors
		$\partial_{s_{i}}\boldsymbol{\gamma}(s)$ span the tangent space,
		\[
		\bigl|\det\boldsymbol{W}(0,s)\bigr|
		=
		\bigl|\boldsymbol{F}(\boldsymbol{\gamma}(s))\cdot\boldsymbol{n}(\boldsymbol{\gamma}(s))\bigr|\,
		\left|
		\det\begin{bmatrix}
			\boldsymbol{n}(\boldsymbol{\gamma}(s)) &
			\partial_{s_{1}}\boldsymbol{\gamma}(s) &
			\cdots &
			\partial_{s_{d-1}}\boldsymbol{\gamma}(s)
		\end{bmatrix}\right|.
		\]
		Let $\boldsymbol{G}(s)$ be the $(d-1)\times(d-1)$ metric matrix with entries
		$G_{ij}(s)=\partial_{s_{i}}\boldsymbol{\gamma}(s)\cdot\partial_{s_{j}}\boldsymbol{\gamma}(s)$.
		Then
		\[
		\left|
		\det\begin{bmatrix}
			\boldsymbol{n}(\boldsymbol{\gamma}(s)) &
			\partial_{s_{1}}\boldsymbol{\gamma}(s) &
			\cdots &
			\partial_{s_{d-1}}\boldsymbol{\gamma}(s)
		\end{bmatrix}\right|
		=
		\sqrt{\det\boldsymbol{G}(s)}.
		\]
		Hence
		\[
		\bigl|\det\boldsymbol{W}(t,s)\bigr|
		=
		\bigl|\det\boldsymbol{W}(0,s)\bigr|
		=
		\bigl|\boldsymbol{F}(\boldsymbol{\gamma}(s))\cdot\boldsymbol{n}(\boldsymbol{\gamma}(s))\bigr|\,
		\sqrt{\det\boldsymbol{G}(s)}.
		\]
		On $\Gamma_{\mathrm{in}}$ we have $\boldsymbol{F}\cdot\boldsymbol{n}<0$, so
		$|\boldsymbol{F}\cdot\boldsymbol{n}|>0$, and by regularity
		$\det\boldsymbol{G}(s)>0$.
		Therefore $\det\boldsymbol{W}(t,s)\neq 0$ on $\mathcal{D}$.
		
		Thus the flow map is a $C^{1}$ local diffeomorphism at every point of
		$\mathcal{D}$. Since it is a bijection,  it is a global
		$C^{1}$-diffeomorphism from $\mathcal{D}$ onto
		$\Omega\setminus\mathcal{N}$.
		
		Finally, apply the change-of-variables theorem on each boundary chart and use
		$\mathrm{d}\sigma(\boldsymbol{x}_{0})=\sqrt{\det\boldsymbol{G}(s)}\,\mathrm{d}s$,
		then sum over charts (using a finite or countable atlas and a partition of unity).
		Since $|\mathcal{N}|=0$, this yields \eqref{eq:change-of-variables}.
	\end{proof}
	\section{Concluding remarks and open questions}
	
	\label{sec:conclusions}
	
	\begin{itemize}
		
		\item \emph{Boundary conditions.}
		The absolute boundary conditions are not used crucially beyond giving a
		meaning to the divergence-free Hodge--Laplacian space and to the inflow
		trace. Other Hodge-type realizations (relative, mixed, or perfect-conductor
		versus insulating walls) could be considered, provided the curl--div system \eqref{eq:div-curl}
		remains well posed and the inflow trace estimate of
		Lemma~\ref{trace-estimate} holds.
		
		\item \emph{Full nonlinear induction equation.}
		The analysis is restricted to the linearized perturbation problem (system \eqref{b-system}) around
		$(\boldsymbol{b},\boldsymbol{v})=(0,0)$. The full induction equation \eqref{eq:induction} is an
		inverse \emph{coefficient} problem rather than an inverse \emph{source}
		problem, and extending the uniqueness and stability results to the
		nonlinear setting even locally, for small data and short times,	remains open.
		
		\item \emph{Stability for internal observations.}
		The quantitative stability estimate is proved only for the final-time
		observation in Theorem \eqref{thm:lip-stab}. Proving a stability estimate for internal observations on
		$(0,T)\times\omega_{\mathrm{obs}}$, and identifying the sharp modulus
		(Lipschitz, H\"older or logarithmic), is an open problem. Whether several
		partial observations can be combined to improve the modulus is also
		unresolved.
		
		\item \emph{Compatibility and admissible data.}
		The curl--div system is overdetermined and requires
		$\mathcal{C}(\boldsymbol{g})=0$ on $\Gamma_{\mathrm{in}}$ (see \eqref{C(g,h)}). In practice the
		inflow trace is measured and may not satisfy this condition exactly.
		Constructing a stable projection onto the admissible set and quantifying
		the error propagation into the reconstruction is an open problem, closely
		related to the feasibility issue in the two-stage numerical method.

		\item \emph{Role of the inflow trace and the gauge.}
		Magnetic observations alone are blind to the velocity component parallel
		to $\boldsymbol{F}$. The inflow trace fixes this gauge, but it would be
		interesting to determine weaker or alternative data (e.g.\ partial traces,
		flux measurements, or additional magnetic observations) that still remove
		the ambiguity.
		
		\item \emph{The separable ansatz.}
		The analysis assumes $\boldsymbol{v}(t,\boldsymbol{x})
		=\alpha(t)\boldsymbol{u}(\boldsymbol{x})$ with a known nonvanishing
		$\alpha\in L^{2}(0,T)$. Relaxing the separability, or treating the time
		profile as unknown, leads to a genuinely time-dependent inverse problem
		that is not covered here.
		
		\item \emph{Time-dependent background fields.}
		The background $\boldsymbol{F}$ is assumed time-independent. In the same manner, our result can be generalized to the case $\boldsymbol{F}(t,\boldsymbol{x})$. 
		
		\item As in \cite{mokhtari2026velocity}, we can treat the case
		$\Omega=\mathbb{R}^{3}$ in the same strategy used in this paper.

	\end{itemize}

	\bibliographystyle{unsrt}
	\bibliography{biblio}
	
\end{document}